\documentclass[10pt]{amsart}
\usepackage[latin1]{inputenc}
\usepackage{amsmath}
\usepackage{amsthm}
\usepackage{amsfonts}
\usepackage{amssymb}
\usepackage[T1]{fontenc}
\usepackage[all]{xy}

\usepackage{calrsfs}
\usepackage{enumerate}
\newtheorem{theorem}{Theorem}[section]
\newtheorem{prop}[theorem]{Proposition}
\newtheorem{lem}[theorem]{Lemma}
\newtheorem{corol}[theorem]{Corollary}

\theoremstyle{definition}
\newtheorem{defi}[theorem]{Definition}

\newtheorem{rmq}[theorem]{Remark}
\newtheorem{exmp}[theorem]{Example}

\def\Gr{{\rm{Gr}}}

\def\d{\partial}

\def\c{{\mathbf{c}}}
\def\k{{\mathbf{k}}}
\def\n{{\mathbf{n}}}
\def\s{{\mathbf{s}}}
\def\x{{\mathbf{x}}}
\def\y{{\mathbf{y}}}
\def\z{{\mathbf{z}}}

\def\add{{\rm{add}\,}}

\def\dim{{\rm{dim}\,}}
\def\ddim{{\mathbf{dim}\,}}
\def\rad{{\rm{Rad}\,}}
\def\rep{{\rm{rep}}}
\def\supp{{\rm{supp}}}

\def\<{\left<}
\def\>{\right>}

\def\ens#1{\left\{ #1 \right\}}
\def\fl{{\longrightarrow}\,}

\def\CC{{\mathcal{C}}}
\def\A{{\mathbb{A}}}

\def\modd{{\textrm{mod-}}}

\def\N{{\mathbb{N}}}
\def\P{{\mathbb{P}}}
\def\Q{{\mathbb{Q}}}
\def\Z{{\mathbb{Z}}}

\def\op{{\rm{op}}}
\def\pp{{\rm{pp}}}
\def\Ext{{\rm{Ext}}}
\def\End{{\rm{End}}}
\def\Hom{{\rm{Hom}}}
\def\Ob{{\rm{Ob}}}

\title{Friezes, Strings and Cluster Variables}
\author{Ibrahim Assem, Gr{\'e}goire Dupont, Ralf Schiffler and David Smith}
\address{Universit\'e de Sherbrooke, Sherbrooke QC, Canada
}
\email{ibrahim.assem@usherbrooke.ca}
\email{gregoire.dupont@usherbrooke.ca}
\address{University of Connecticut, Storrs CT, USA
}
\email{schiffler@math.uconn.edu}
\address{Bishop's University, Sherbrooke QC, Canada
}
\email{dsmith@ubishops.ca}

\date{\today}

\begin{document}

\begin{abstract}
	To any walk in a quiver, we associate a Laurent polynomial. When the walk is the string of a string module over a 2-Calabi-Yau tilted algebra, we prove that this Laurent polynomial coincides with the corresponding cluster character of the string module, up to an explicit normalising monomial factor.
\end{abstract}

\maketitle

% \setcounter{tocdepth}{1}
% \tableofcontents

\section*{Introduction}
	In the early 2000, S. Fomin and A. Zelevinsky introduced the class of cluster algebras with the purpose of building a combinatorial framework for studying total positivity in algebraic groups and canonical bases in quantum groups, see \cite{cluster1}. Since then, the study of cluster algebras was shown to be connected to several areas of mathematics, notably combinatorics, Lie theory, Poisson geometry, Teichm\"{u}ller theory, mathematical physics and representation theory of algebras.

	A cluster algebra is a commutative algebra generated by a set of variables, called \emph{cluster variables}, obtained recursively by a combinatorial process known as \emph{mutation}, starting from an initial set of cluster variables, the \emph{initial cluster}, and a quiver without cycles of length at most two. One of the most remarkable facts about cluster variables is that they can be expressed as Laurent polynomials in terms of the initial cluster variables \cite{cluster1}; this is the so-called \emph{Laurent phenomenon}. Also, it is conjectured that the coefficients in this expression are always non-negative; this is the \emph{positivity conjecture}. The problem of computing explicitly the cluster variables is a difficult one and has been extensively studied. The most general results known at the present time are in \cite{FK,Plamondon:ClusterAlgebras, MSW:positivity}.

	In order to compute cluster variables, one may use \emph{friezes}. Friezes, which go back to works of Coxeter and Coxeter-Conway \cite{C71, CoxeterConway1, CoxeterConway2}, are an efficient combinatorial tool which mimics the application of mutations on sinks or sources of the given quiver, hence an obvious combinatorial connection with cluster algebras \cite{CarollPrice, CC, Propp:frieze, Musiker:classical, ARS:frises}. It is now known that connections between friezes and cluster algebras are deeper than just combinatorics, see for instance \cite{Dupont:stabletubes,AD:algorithm,KS:frieze, BM10, FM10}. Our starting point for the present paper was the result in \cite{ARS:frises} giving an explicit formula as a product of $2\times 2$ matrices for all cluster variables in coefficient-free cluster algebras of type $\A$ and all but finitely many cluster variables in coefficient-free cluster algebras of type $\widetilde{\A}$, thus explaining at the same time the Laurent phenomenon and positivity. Our objective here is to show that the same technique can be used for computing the cluster variables associated with the string modules over a 2-Calabi-Yau tilted algebra (in the sense of \cite{Reiten:CIRM07}).

	Besides friezes, our second main tool is the notion of a \emph{cluster character}. In \cite{CC}, Caldero and\ Chapoton noticed that cluster variables
	in simply-laced coefficient-free cluster algebras of finite type can be expressed as generating series of Euler-Poincar\'{e} characteristics of Grassmannians of submodules. Generalising this work, Caldero-Keller \cite{CK2}, Palu \cite{Palu} and Fu-Keller \cite{FK} introduced the notion of a cluster character associating to each module $M$ over a 2-Calabi-Yau tilted algebra $B_{T}$ a certain Laurent polynomial $X_{M}^{T}$ allowing one to compute a corresponding cluster variable. In general, cluster characters are hard to compute because one first needs to find the Euler characteristics of Grassmannians of submodules, and then the dimensions of certain Hom-spaces in the corresponding 2-Calabi-Yau category.

	One class of algebras, however, whose representation theory is reasonably well-understood is the class of string algebras, introduced by Butler and Ringel in \cite{ButlerRingel:strings} (see also \cite{WW85}). In particular, indecomposable modules over string algebras are partitioned into two sets: string and band modules, and only string modules can be associated with cluster variables. The Euler characteristics of Grassmannians of submodules of string modules were computed by Cerulli and Haupt \cite{Cerulli:string, Haupt:string}. Nevertheless, their methods do not allow one to compute explicitly the associated cluster character.

	The main result of this paper gives an explicit formula for the cluster character associated with a string module over a 2-Calabi-Yau tilted algebra.
	This can be stated as follows. To any walk $c$ in a locally finite quiver $Q$, we associate a Laurent polynomial $L_{c}$ in the ring of Laurent polynomials in the indeterminates $x_{i}$ indexed by the set $Q_0$ of points of $Q$, which can be expressed as a product of $2\times 2$ matrices as in \cite{ARS:frises} (see Section \ref{ssection:formula} below). Now, we let $T$ be a tilting object in a Hom-finite triangulated 2-Calabi-Yau category $\mathcal{C}$ and $B_{T}=$ $\End_{\CC}(T)$ be the corresponding 2-Calabi-Yau tilted algebra whose ordinary quiver is denoted by $Q$. Moreover, to any string $B_{T}$-module $M$, we associate a tuple of integers $\n_M=(n_i)_{i\in Q_0}$ which we call the \emph{normalisation} of $M$ (see Section \ref{ssection:normalisation}), and the Laurent polynomial $L_{M}$ which is just the Laurent polynomial $L_{c}$ attached to the string $c$ of $M$ in the quiver $Q$. Using the notation $\mathbf{x}^{\mathbf{n}_{M}}=\prod_{i\in Q_{0}}x_{i}^{n_i}$, our main result (Theorem \ref{theorem:main} below) can be stated as saying that~:
	$$X_{M}^{T}=\frac{1}{\mathbf{x}^{\mathbf{n}_{M}}}L_{M}.$$

	This result entails several interesting consequences. We first obtain the positivity of the Laurent polynomial $X_{M}^{T}$ for any string $B_{T}$-module $M$ (see Corollary \ref{corol:positivity}), thus reproving a result of Cerulli and Haupt \cite{Cerulli:string, Haupt:string}. Our results also apply to the cases of string modules over cluster-tilted algebras, string modules over gentle algebras arising from unpunctured surfaces (see \cite{ABCP}) and, more generally, string modules over finite dimensional Jacobian algebras associated with quivers with potentials in the sense of \cite{DWZ:potentials}. We also obtain a new proof of the positivity conjecture for cluster algebras arising from surfaces without punctures, see \cite{ST:unpunctured, S:unpunctured2, MSW:positivity}.

	The paper is organised as follows. Section \ref{section:formula} introduces the basic definitions and presents our combinatorial formula. Section \ref{section:realisabledata} introduces the concept of \emph{realisable quadruples} which is the context in which our formula will actually compute cluster characters. This is closely related to the notion of triangulated 2-Calabi-Yau realisation in the sense of Fu and Keller \cite{FK}. Section \ref{section:clusterchar} recalls all the necessary background from \cite{Palu,FK} concerning cluster characters. Sections \ref{section:typeA} and \ref{section:main} contain the proof of our main result. In Section \ref{section:typeA}, we actually prove a weaker version of our theorem which will be used in order to prove the general case in Section \ref{section:main}. In Section \ref{section:applications} we present possible applications of the results to the study of positivity in cluster algebras. In Section \ref{section:morenormalisation}, we investigate the normalising factor and explicitly compute it for several cases of string modules over cluster-tilted algebras. The last section presents some detailed examples.

\section{The matrix formula}\label{section:formula}
	\subsection{Notations}
		Throughout the article, $\k$ denotes an algebraically closed field. Given a quiver $Q$, we denote by $Q_0$ its set of points and by $Q_1$ its set of arrows. For any arrow $\alpha \in Q_1$, we denote by $s(\alpha)$ its \emph{source} and by $t(\alpha)$ its \emph{target}. We sometimes simply write $\alpha: s(\alpha) \to t(\alpha)$ or $s(\alpha) \xrightarrow{\alpha} t(\alpha)$. For any point $i \in Q_0$, we set
		$$Q_1(i,-) = \ens{\alpha \in Q_1 \, | \, s(\alpha) = i} \textrm{, } Q_1(-,i)= \ens{\alpha \in Q_1 \, | \, t(\alpha) = i}$$
		and if $F$ is a subset of points in $Q_0$, we set
		$$Q_1(i,F) = \ens{\alpha \in Q_1 \, | \, s(\alpha) = i, \, t(\alpha) \in F} \textrm{, } Q_1(F,i) = \ens{\alpha \in Q_1 \, | \, s(\alpha) \in F, \, t(\alpha) = i}$$
		and finally, for any $i,j \in Q_0$, we set $Q_1(i,j) = Q_1(i,-) \cap Q_1(-,j)$.

		To any quiver $Q$, we associate a family $\x_Q=\ens{x_i|i \in Q_0}$ of indeterminates over $\Z$. We let $\mathcal L(\x_Q) = \Z[x_i^{\pm 1}\, |\, i \in Q_0]$ be the ring of Laurent polynomials in the variables $x_i$, with $i \in Q_0$ and $\mathcal F(\x_Q) = \Q(x_i \, |\, i \in Q_0)$ be the field of rational functions in the variables $x_i$, with $i \in Q_0$. For any $\mathbf d=(d_i)_{i \in Q_0} \in \Z^{Q_0}$, we set $\x_Q^\mathbf d = \prod_{i \in Q_0} x_i^{d_i}$.

		A \emph{bound quiver} is a pair $(Q,I)$ such that $Q$ is a finite quiver (that is, $Q_0$ and $Q_1$ are finite sets) and $I$ is an admissible ideal in the path algebra $\k Q$ of $Q$. Given a finite dimensional basic $\k$-algebra $B$, there exists a bound quiver $(Q,I)$ such that $B \simeq \k Q/ I$ and the quiver $Q$ is called the \emph{ordinary quiver} of $B$ (see for instance \cite{ASS}). We always identify the category mod-$B$ of finitely generated right $B$-modules with the category $\rep(Q,I)$ of $\k$-representations of $Q$ satisfying the relations in $I$. For a $B$-module $M$, we denote by $M(i)$ the $\k$-vector space at the point $i \in Q_0$ and $M(\alpha)$ the $\k$-linear map at the arrow $\alpha \in Q_1$.
		
% 		If $B$ is the path algebra $\k Q$ of a quiver, the quiver $Q$ will sometimes be identified with the bound quiver $(Q,(0))$ and the category $\rep(Q,(0))$ is simply denoted by $\rep(Q)$.

	\subsection{Walks and strings}
		Let $Q$ be a quiver. For any arrow $\beta \in Q_1$, we denote by $\beta^{-1}$ a \emph{formal inverse} for $\beta$, with $s(\beta^{-1})=t(\beta)$, $t(\beta^{-1})=s(\beta)$ and we set $(\beta^{-1})^{-1} = \beta$. % We denote $Q_1^{\op}$ the set of all formal inverse of arrows in $Q_1$. A \emph{formal arrow} is an element $\beta \in Q_1 \sqcup Q_1^{\op}$.

		A \emph{walk} of length $n \geq 1$ in $Q$ is a sequence $c=c_1 \cdots c_n$ where each $c_i$ is an arrow or a formal inverse of an arrow and such that $t(c_i) = s(c_{i+1})$ for any $i \in \ens{1, \ldots, n-1}$. The \emph{source of the walk} $c$ is $s(c)=s(c_1)$ and the \emph{target of the walk} $c$ is $t(c) = t(c_n)$. With any point $i \in Q_0$, we associate a walk $e_i$ of length zero which is given by the stationary path at point $i$. For any walk $c$ on $Q$, we denote by $c_0$ the walk of length zero $c_0 = e_{s(c)}$.

		If $(Q,I)$ is a bound quiver, a \emph{string} in $(Q,I)$ is either a walk of length zero or a walk $c=c_1 \cdots c_n$ of length $n \geq 1$ such that $c_i \neq c_{i+1}^{-1}$ for any $i \in \ens{1, \ldots, n-1}$ and such that no walk of the form $c_i c_{i+1} \cdots c_t$ nor its inverse belongs to $I$ for $1 \leq i$ and $t \leq n$. If $Q$ is a quiver, a \emph{string in the quiver} $Q$ is a string in the bound quiver $(Q,(0))$. If $B \simeq \k Q/I$ is a finite dimensional $\k$-algebra and $c$ is a string in $(Q,I)$, then we also say that $c$ is \emph{a string in $B$}.
	
		To any string in a finite dimensional $\k$-algebra $B$, we can naturally associate an indecomposable finite dimensional right $B$-module $M_c$, called \emph{string module} as in \cite[\S 3]{ButlerRingel:strings}. Namely, if $B$ has ordinary quiver $Q$ and $c$ is a string, we define $M_c$ as follows. If $c=e_v$ has length zero, then $M_c$ is the simple representation $S_v$ at the point $v \in Q_0$. Otherwise, $c$ is a string of length $n \geq 1$ and we write $c = \xymatrix{v_1 \ar@{-}[r]^{c_1} & \cdots \ar@{-}[r]^{c_n} & v_{n+1}}$. For any $v \in Q_0$, we set $I_v = \ens{ i \in [1,n+1] \ | \ v_i = v}$ and we define $M_c(v)$ as  the $|I_v|$-dimensional $\k$-vector space with basis $z_i$, with $i \in I_v$. For any $1 \leq i \leq n$, if $c_i = \beta \in Q_1$, we set $M_c(\beta)(z_{i-1}) = z_i$ and if  $c_i = \beta \in Q_1^{-1}$, we set $M_c(\beta)(z_{i}) = z_{i-1}$. Finally, if $\alpha: v \fl v'$ is an arrow in $Q_1$ and if $z_j$ is one of the basis vector of $M_c(v)$ such that $M_c(\alpha)(z_j)$ is not yet defined, we set $M_c(\alpha)(z_j)=0$. 

		A string module is also called a \emph{string representation} of the corresponding bound quiver. For any string module $M$, we denote by $\s(M)$ the corresponding string.

% 		In \cite[\S 3]{ButlerRingel:strings}, the authors introduced a particular class of finite dimensional algebras, called \emph{string algebras}. For string algebras, every indecomposable module is either a string module or a band module. Moreover, every \emph{rigid} module, that is every module without self-extension, is a string module. In the representation-theoretical approach to cluster algebras, indecomposable rigid modules over cluster-tilted algebras, and more generally 2-Calabi-Yau tilted algebras, are of particular interest (see Section \ref{section:clusterchar} for details). In this context, an important class of string algebras is the class of Jacobian algebra associated with triangulations of unpunctured surfaces (these are in fact a gentle algebras) \cite{ABCP}. For example cluster-tilted algebras of Dynkin type $\A$ or affine type $\widetilde \A$, which arise as Jacobian algebras associated with triangulations of the disc or the annulus, are string algebras.

	\subsection{A formula for walks}\label{ssection:formula}
		For any locally finite quiver $Q$, we define a family of matrices with coefficients in $\Z[\x_Q] = \Z[x_i \, | \, i \in Q_0]$ as follows.

		For any arrow $\beta \in Q_1$, we set
		$$A(\beta)=\left[\begin{array}{cc}
	               x_{t(\beta)} & 0 \\
			1 & x_{s(\beta)}
		\end{array}\right]
		\textrm{ and }
		A(\beta^{-1})=\left[\begin{array}{cc}
			x_{t(\beta)} & 1 \\
				0 & x_{s(\beta)}
			\end{array}\right].$$

		Let $c=c_1 \cdots c_n$ be a walk of length $n \geq 1$ in $Q$. For any $i \in \ens{0, \ldots, n}$ we set
		$$v_{i+1} = t(c_i)$$
		(still with the notation $c_0 = e_{s(c)}$) and
		$$V_{c}(i)=\left[\begin{array}{cc}
			\displaystyle \prod_{\substack{\alpha \in Q_1(v_i,-) \\ \alpha \neq c_i^{\pm 1},c_{i-1}^{\pm 1}}} x_{t(\alpha)} & 0 \\
			0 & \displaystyle \prod_{\substack{\alpha \in Q_1(-,v_i) \\ \alpha \neq c_i^{\pm 1},c_{i-1}^{\pm 1}}} x_{s(\alpha)}\\
		\end{array}\right].$$
% 		In other words, the product in the upper-left entry is taken over all the arrows starting at $v_i$ except those which are (or whose inverse are) are used in the string $c$ to ``arrive'' at the point $v_i$ and to ``leave'' the point $v_i$. The product is the lower-right entry is defined similarly.

		We then set
		$$L_c = \frac{1}{x_{v_1} \cdots x_{v_{n+1}}} \left[\begin{array}{cc} 1,1 \end{array}\right] V_{c}(1) \left( \prod_{i=1}^n A(c_i) V_{c}({i+1})\right) \left[\begin{array}{c}1 \\ 1 \end{array}\right] \in \mathcal L(\x_Q).$$

		If $c=e_i$ is a walk of length 0 at a point $i$, we similarly set
		$$V_{e_i}(1)=\left[\begin{array}{cc}
			\displaystyle \prod_{\alpha \in Q_1(i,-)} x_{t(\alpha)} & 0 \\
			0 & \displaystyle \prod_{\alpha \in Q_1(-,i)} x_{s(\alpha)}\\
		\end{array}\right].$$
		and
		$$L_{e_i}=\frac{1}{x_i} \left[\begin{array}{cc} 1,1 \end{array}\right] V_{e_i}(1) \left[\begin{array}{c}1 \\ 1 \end{array}\right] \in \mathcal L(\x_Q).$$	

		In other words, if $c$ is any walk, either of length zero, or of the form $c=c_1 \cdots c_n$, we have
		\begin{equation}\label{eq:Lc}
			L_c = \frac{1}{\prod_{i=0}^n x_{t(c_i)}} \left[\begin{array}{cc} 1,1 \end{array}\right] \left( \prod_{i=0}^n A(c_i) V_{c}(i+1)\right) \left[\begin{array}{c}1 \\ 1 \end{array}\right] \in \mathcal L(\x_Q)
		\end{equation}
		with the convention that $A(c_0)$ is the identity matrix.

		\begin{exmp}
			Consider the quiver
			$$
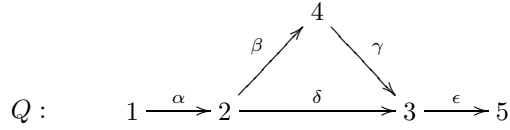
\xymatrix{
				&&& 4 \ar[rd]^{\gamma} \\
				Q: & 1 \ar[r]^{\alpha} & 2 \ar[ru]^{\beta} \ar[rr]^{\delta} && 3 \ar[r]^{\epsilon} & 5
			}$$
			and consider the path $c = \delta^{-1}\beta\gamma$ in $Q$. Then,
			\begin{align*}
			L_c 	
				& = \frac{1}{x_2 x_3^2 x_4} \left(
				\left[\begin{array}{cc} 1,1 \end{array}\right]
				\left[\begin{array}{cc} x_5 & 0 \\ 0 & x_4 \end{array}\right]
				\left[\begin{array}{cc} x_3 & 1 \\ 0 & x_2 \end{array}\right]
				\left[\begin{array}{cc} 1 & 0 \\ 0 & x_1 \end{array}\right]
				\right.\\
				& \left.
				\left[\begin{array}{cc} x_4 & 0 \\ 1 & x_2 \end{array}\right]
				\left[\begin{array}{cc} 1 & 0 \\ 0 & 1 \end{array}\right]
				\left[\begin{array}{cc} x_3 & 0 \\ 1 & x_4 \end{array}\right]
				\left[\begin{array}{cc} x_5 & 0 \\ 0 & x_2 \end{array}\right]
				\left[\begin{array}{c}1 \\ 1 \end{array}\right] \right)\\
				& = \frac{x_1x_2^3x_4^2 + 2x_1x_2^2x_4x_5 + x_1x_2x_3x_4x_5 + x_3^2x_4x_5^2 + x_1x_2x_5^2 + x_1x_3x_5^2}{x_2x_3^2x_4}
			\end{align*}
		\end{exmp}

	\subsection{A formula for string modules}
		\begin{defi}
			Let $B$ be a finite dimensional $\k$-algebra with bound quiver $(Q,I)$ and let $M$ be a string $B$-module with corresponding string $\s(M)$. We set
			$$L_M = L_{\s(M)} \in \mathcal L(\x_Q).$$
		\end{defi}

		If $Q$ is a quiver, a \emph{subquiver} $R$ of $Q$ is a quiver $R$ such that $R_0 \subset Q_0$ and such that for any $i,j \in R_0$, the set of arrows from $i$ to $j$ in $R_1$ is a subset of the set of arrows from $i$ to $j$ in $Q_1$. If $R$ is a subquiver of $Q$, we naturally identify $\mathcal L(\x_R)$ with a subring of $\mathcal L(\x_Q)$.

		A subquiver $R$ of $Q$ is called a \emph{full subquiver} if for all $i,j \in R_0$ the set of arrows from $i$ to $j$ in $R_1$ equals the set of arrows from $i$ to $j$ in $Q_1$. If $B$ is a finite dimensional $\k$-algebra with bound quiver $(Q,I)$ and $M$ is a $B$-module, the \emph{support} of $M$ is the full subquiver $\supp(M)$ of $Q$ consisting of the points $i \in Q_0$ such that $M(i) \neq 0$. The \emph{closure of the support of $M$} is the full subquiver $\overline{\supp(M)}$ of $Q$ formed by the points $i \in Q_0$ which are in the support of $M$ or such that there exists an arrow $\alpha$ such that $s(\alpha) \in \supp(M)_0$ and $t(\alpha)=i$ or $t(\alpha) \in \supp(M)_0$ and $s(\alpha)=i$.

		With these identifications, if $B$ is a finite dimensional $\k$-algebra and $M$ is a string $B$-module, then
		$$L_M \in \mathcal L(\x_{\overline{\supp(M)}}).$$

\section{Realisable quadruples}\label{section:realisabledata}
	In the previous section we associated with any string module $M$ over a finite dimensional algebra $B$ a certain Laurent polynomial $L_M$. In this section we provide a context in which the algebra $B$ arises in connection with some cluster algebras so that we can compare the Laurent polynomials $L_M$ with cluster variables or, more generally, with cluster characters. The context in which we work is the context of triangulated 2-Calabi-Yau realisations introduced in \cite{FK}.

	\subsection{Definitions}
		An \emph{ice quiver} is a pair $(\mathcal Q,F)$ such that $\mathcal Q$ is a finite connected quiver without loops and 2-cycles and $F$ is a (possibly empty) subset of points of $\mathcal Q_0$, called \emph{frozen points}, such that there are no arrows between points in $F$. The \emph{unfrozen part} of $(\mathcal Q,F)$ is the full subquiver of $\mathcal Q$ obtained by deleting the points in $F$.

		For any quiver $Q$, we denote by $B(Q) =(b_{ij}) \in M_{Q_0}(\Z)$ the skew-symmetric matrix defined for any $i,j \in Q_0$ by
		$$b_{ij} = |Q_1(i,j)| - |Q_1(j,i)|.$$
% 		that is, $b_{ij}$ equals the difference between the number of arrows from $i$ to $j$ in $Q_1$ and the number of arrows from $j$ to $i$ in $Q_1$.

		If $(\mathcal Q,F)$ is an ice quiver with unfrozen part $Q$, then we can fix an ordering of the points in $\mathcal Q_0$ such that
		$$B(\mathcal Q) =
		\left[\begin{array}{cc}
			B(Q) & -C^t \\
			C & 0_{F \times F}
		\end{array}\right]$$
		where $C$ is a matrix in $M_{F \times Q_0}(\Z)$.

		For any ice quiver $(\mathcal Q,F)$ with unfrozen part $Q$, we denote by $\mathcal A(\mathcal Q,F)$ the cluster algebra of geometric type with initial seed $(\widetilde B(\mathcal Q),\mathbf x, \mathbf y)$ where
		$$\widetilde B(\mathcal Q) =
		\left[\begin{array}{c}
			B(Q) \\
			C
		\end{array}\right],$$
		$\mathbf x=(x_i, i \in Q_0)$ and $\mathbf y=(x_i, i \in F)$ (see \cite{cluster4}).

		A $\k$-linear category $\CC$ is called \emph{Hom-finite} if $\Hom_{\CC}(M,N)$ is a finite dimensional $\k$-vector space for any two objects $M,N$ in $\CC$. A $\k$-linear triangulated category $\CC$ is called \emph{$2$-Calabi-Yau} if there is a bifunctorial isomorphism
		$$\Hom_{\CC}(X,Y) \simeq D \Hom_{\CC}(Y, X[2])$$
		where $D=\Hom_{\k}(-,\k)$ is the standard duality and $[1]$ denotes the suspension functor.

		For any objects $M,N$ in $\CC$, we denote by $\Ext^1_{\CC}(M,N)$ the space $\Hom_{\CC}(M,N[1])$. An object $M$ in the category $\CC$ is called \emph{rigid} if $\Ext^1_{\CC}(M,M)=0$. An object $T$ in $\CC$ is called a \emph{tilting object} in $\CC$ if for any object $X$ in $\CC$, the equality $\Ext^1_{\CC}(T,X) = 0$ is equivalent to the fact that $X$ belongs to the additive subcategory $\add T$ of $\CC$. Note that these objects are sometimes called \emph{cluster-tilting} in the literature. It is known that the combinatorics of cluster algebras are closely related to the combinatorics of tilting objects in triangulated 2-Calabi-Yau categories \cite{BIRS}.

		\begin{defi}
			A \emph{realisable quadruple} is a quadruple $(\mathcal Q, F,\CC,T)$ such that~:
			\begin{enumerate}[(a)]
				\item $(\mathcal Q,F)$ is an ice quiver~;
				\item $\CC$ is a Hom-finite triangulated 2-Calabi-Yau category whose tilting objects form a cluster structure in the sense of \cite{BIRS}~;
				\item $T$ is a tilting object in $\CC$~;
				\item the ordinary quiver of $B_T=\End_{\CC}(T)$ is $\mathcal Q$.
			\end{enumerate}
		\end{defi}
		Following Reiten \cite{Reiten:CIRM07}, every algebra of the form $\End_{\CC}(T)$ as above is called a \emph{2-Calabi-Yau tilted algebra}. In order to simplify terminology, a category $\CC$ satisfying (b) is simply called a \emph{triangulated 2-Calabi-Yau category}.

		\begin{defi}
			Given a realisable quadruple $(\mathcal Q, F, \CC,T)$, a $B_T$-module $M$ is called~:
			\begin{enumerate}[a)]
				\item \emph{unfrozen} if $M(i) \neq 0$ implies $i \not \in F$~;
				\item \emph{unfrozen indecomposable} if it is unfrozen and indecomposable as a $B_T$-module~;
				\item \emph{unfrozen sincere} if $M(i) \neq 0$ if and only if $i \not \in F$.
			\end{enumerate}
		\end{defi}

	\subsection{Examples}
		The notion of realisable quadruple covers a lot of situations in the context of cluster algebras. We now list some examples of such situations.

		\begin{exmp}\label{exmp:realacyclic}
			Let $Q$ be an acyclic quiver and let $\CC$ be the \emph{cluster category} of $Q$, first defined in \cite{BMRRT} (see \cite{CCS1} for an alternative description in Dynkin type $\A$). It is canonically equipped with a structure of triangulated 2-Calabi-Yau category \cite{K} and for any tilting object $T$ in $\CC$, an algebra of the form $\End_{\CC}(T)$ is called a \emph{cluster-tilted algebra}, as defined in \cite{BMR1} (see also \cite{CCS1}). Thus, if $Q_T$ is the ordinary quiver of the cluster-tilted algebra $\End_{\CC}(T)$, the quadruple $(Q_T,\emptyset, \CC,T)$ is realisable and every $B_T$-module is unfrozen.
		\end{exmp}

		\begin{exmp}\label{exmp:QP}
			Let $(Q,W)$ be a Jacobi-finite quiver with potential, that is, a quiver with potential in the sense of \cite{DWZ:potentials} such that the Jacobian algebra $\mathcal J_{(Q,W)}$ is finite dimensional. Let $\CC_{(Q,W)}$ be the generalised cluster category constructed by Amiot \cite[\S 3]{Amiot:clustercat}. Then $\CC_{(Q,W)}$ is a triangulated 2-Calabi-Yau category and there exists a tilting object $T$ in $\CC_{(Q,W)}$ such that $(Q,\emptyset,\CC_{(Q,W)},T)$ is a realisable quadruple \cite[Theorem 3.6]{Amiot:clustercat}.

			This example is of particular interest for the study of cluster algebras arising from surfaces in the sense of \cite{FST:surfaces} (see also Section \ref{ssection:positivity} for more details). Indeed, Labardini associated a non-degenerate Jacobi-finite quiver with potential $(Q,W)$ to any marked surface $(S,M)$ with non-empty boundary \cite{Labardini:potentialssurfaces}. Thus, the generalised cluster category $\CC_{(Q,W)}$ provides a categorification for the cluster algebra $\mathcal A(S,M)$ associated with the surface. Moreover, if the surface is \emph{unpunctured}, that is, if there are no marked points in the interior of the surface, it is known that the Jacobian algebra $\mathcal J_{(Q,W)}$ is a string algebra (it is in fact a gentle algebra, see \cite{ABCP}). The $\mathcal J_{(Q,W)}$-modules without self-extension are thus string modules and it follows from \cite{FK} that the cluster variables in $\mathcal A(S,M)$ can be studied via cluster characters associated with string $\mathcal J_{(Q,W)}$-modules (see Section \ref{section:clusterchar} for details).
		\end{exmp}

		\begin{exmp}
			Let $A$ be a finite dimensional $\k$-algebra of global dimension 2 and let $\widetilde Q$ be the ordinary quiver of the relation extension of $A$, that is, the trivial extension of $A$ by the $A$-bimodule $\Ext^2_A(DA,A)$ (see for instance \cite{ABS:ext}). Then Amiot also associated with $A$ a generalised cluster category $\CC_A$ and proved that there exists a tilting object $T$ in $\CC_A$ such that $(\widetilde Q,\emptyset,\CC_A,T)$ is a realisable quadruple \cite[Theorem 4.10]{Amiot:clustercat}.
			% We refer to \cite[\S 4]{Amiot:clustercat} for more details on this generalised cluster category.
		\end{exmp}

		\begin{exmp}\label{exmp:Qpp}
			We now give a fundamental example with a non-empty set of frozen points. For any quiver $Q$, we denote by $(Q^{\pp},Q_0')$ its \emph{principal extension}. It is defined as follows. We fix a copy $Q_0'=\ens{i' \, | \, i \in Q_0}$ of $Q_0$ and we set $Q^\pp_0 = Q_0 \sqcup Q_0'$. The arrows in $Q^\pp_1$ between two points in $Q_0$ are the same as in $Q$ and for any $i \in Q_0$ we add an extra arrow $i' \fl i$ in $Q^\pp_1$. Thus, the matrix of $Q^\pp$ is given by
			$$B(Q^\pp) =
			\left[\begin{array}{cc}
				B(Q) & -I_{Q_0'} \\
				I_{Q_0'} & 0
			\end{array}\right]$$
			where $I_{Q_0'}$ denotes the identity matrix in $M_{Q_0'}(\Z)$.
			
			Now, if $Q$ is the ordinary quiver of a cluster-tilted algebra $B_T$, it is known that there exists a 2-Calabi-Yau category $\CC^\pp$, endowed with a tilting object $T^\pp$, obtained via a process of \emph{principal gluing}, such that $(Q^\pp,Q_0',\CC^\pp,T^\pp)$ is a realisable quadruple for which every $B_T$-module can naturally be viewed as an unfrozen module (see \cite[\S 6.3]{FK} or Corollary \ref{corol:clustertilted}).
		\end{exmp}

	\subsection{Blown-up ice quivers and their realisations}
		In the proof of the main theorem of this article, we are interested in a particular family of ice quivers, called \emph{blown-up}. We now give some details concerning these quivers.

		\begin{defi}
			We say that an ice quiver $(\mathcal Q,F)$ is \emph{blown-up} if $|\mathcal Q_1(i,-) \cup \mathcal Q_1(-,i)| \leq 1$ for every point $i \in F$, that is, if there exists at most one arrow starting or ending at any frozen point.
		\end{defi}

		\begin{exmp}
			\begin{enumerate}[a)]
				\item Any ice quiver with an empty set of frozen points is blown-up.
				\item If $Q$ is any quiver, then its principal extension $(Q^\pp,Q_0')$ defined in Example \ref{exmp:Qpp} is blown-up.
			\end{enumerate}
		\end{exmp}

		\begin{rmq}
			Any blown-up ice quiver $(\mathcal Q,F)$ whose unfrozen part is acyclic can be embedded in a realisable quadruple $(\mathcal Q,F,\CC,T)$. Indeed, since $(\mathcal Q,F)$ is blown-up with an acyclic unfrozen part, $\mathcal Q$ is also acyclic. Let thus $\CC$ be the cluster category of the quiver $\mathcal Q$. Then the path algebra $\k \mathcal Q$ is identified with a tilting object in $\CC$ and the corresponding cluster-tilted algebra is isomorphic to $\k \mathcal Q$ so that its ordinary quiver is $\mathcal Q$. Thus $(\mathcal Q,F,\CC,\k \mathcal Q)$ is a realisable quadruple.
		\end{rmq}

		We can construct a wide class of examples of realisable quadruples with the following proposition, due to Amiot~:
		\begin{prop}[\cite{Amiot:clustercat}]\label{prop:glueingQP}
			If $\CC_1$ and $\CC_2$ are two generalised cluster categories associated with Jacobi-finite quivers with potentials $(Q_1,W_1)$ and $(Q_2,W_2)$, then for any matrix $C$ with non-negative integer entries, there exists a Jacobi-finite quiver with potential $(Q',W')$ whose generalised cluster category $\CC_{(Q',W')}$ has a tilting object $T'$ such that the matrix associated with the ordinary quiver of the 2-Calabi-Yau tilted algebra $\End_{\CC_{(Q',W')}}(T')$ is
			$$B' = \left[\begin{array}{cc}
				B(Q_1) & -C^t \\
				C & B(Q_2)
			\end{array}\right].$$
		\end{prop}

		\begin{corol}\label{corol:glueing}
			Let $(\mathcal Q,F)$ be a blown-up ice quiver with unfrozen part $Q$. Assume that there exists a potential $W$ on $Q$ such that the Jacobian algebra $\mathcal J_{(Q,W)}$ is finite dimensional. Then there exists a triangulated 2-Calabi-Yau category $\CC$ and a tilting object $T$ in $\CC$ such that $(\mathcal Q,F,\CC,T)$ is a realisable quadruple.
		\end{corol}
		\begin{proof}
			We let $\CC^{(0)}$ be the generalised cluster category associated with the quiver with potential $(Q,W)$. According to Example \ref{exmp:QP}, there exists a tilting object $T^{(0)}$ in $\CC$ such that the quadruple $(Q,\emptyset,\CC^{(0)},T^{(0)})$ is realisable.

			In order to simplify notations, we write $F = \ens{1, \ldots, m}$. For any $i \in F$, we denote by $\CC_{\A_i}$ a copy of the cluster category of type $\A_1$, which is a particular case of generalised cluster category in the sense of Amiot. We now construct, by induction on $i \in \ens{1, \ldots, m}$, a category $\CC^{(i)}$ with a tilting object $T^{(i)}$ such that the ordinary quiver of the 2-Calabi-Yau tilted algebra $\End_{\CC^{(i)}}(T^{(i)})$ is the full subquiver $(\mathcal Q,F)$ formed by points in $Q$ and frozen points in $\ens{1, \ldots, i}$.

			Since $(\mathcal Q,F)$ is blown-up, for every $i \in \ens{1, \ldots, m}$, there exists a unique arrow $\alpha_i \in \mathcal Q_1(i,-) \sqcup \mathcal Q_1(-,i)$. If $\alpha_i \in \mathcal Q_1(i,-)$, we construct the category $\CC^{(i)}$ gluing $\CC^{(i-1)}$ and $\CC_{\A_i}$ as provided by Proposition \ref{prop:glueingQP} with $\CC_1 = \CC^{(i-1)}$ and $\CC_2 = \CC_{\A_i}$ and we denote by $T^{(i)}$ the canonical tilting object in this gluing. If $\alpha_i \in \mathcal Q(-,i)$, we construct a gluing $\CC^{(i)}$ of $\CC_{\A_i}$ and $\CC^{(i-1)}$ as provided by Proposition \ref{prop:glueingQP} with $\CC_1 = \CC_{\A_i}$ and $\CC_2 = \CC^{(i-1)}$ and we denote by $T^{(i)}$ the canonical tilting object in this gluing. We finally set $\widetilde {\CC} = \CC^{(m)}$ and $\widetilde T = T^{(m)}$.
		\end{proof}

		\begin{corol}\label{corol:clustertilted}
			Let $(\mathcal Q,F)$ be a blown-up ice quiver whose unfrozen part is the ordinary quiver of a cluster-tilted algebra. Then $(\mathcal Q,F)$ can be embedded in a realisable quadruple $(\mathcal Q,F,\CC,T)$.
		\end{corol}
		\begin{proof}
			It is proved in \cite{Keller:deformed} and \cite{BIRSm} that any cluster-tilted algebra is the Jacobian algebra of a Jacobi-finite quiver with potential. The result thus follows from Corollary \ref{corol:glueing}.
		\end{proof}

\section{Cluster characters with coefficients}\label{section:clusterchar}
	In this section, we collect some background concerning Fu-Keller's cluster characters \cite{FK}. These characters allow one to give an explicit realisation of the elements in a cluster algebra in terms of the geometry and the homology underlying a 2-Calabi-Yau category.

	We fix a triangulated 2-Calabi-Yau category $\CC$ with suspension functor $[1]$ and we fix a tilting object $T$ in $\CC$.

	We let $\mathcal Q$ be the ordinary quiver of the 2-Calabi-Yau tilted algebra $B_T = \End_{\CC}(T)$. Indecomposable direct summands of $T$ label the points in $\mathcal Q_0$. Let $F$ be a set of frozen points in $\mathcal Q_0$. We set $T_F = \bigoplus_{i \in F}T_i$. Consider the full subcategory $\mathcal U$ of $\CC$ formed by the objects $X$ such that $\Hom_{\CC}(T_F,X)=0$. Let $(\add T[1])$ be the ideal consisting of those morphisms factoring through objects of $\add T[1]$.

	\begin{theorem}[\cite{BMR1,KR:clustertilted}]
		The functor $\Hom_{\CC}(T,-)$ induces an equivalence
		$$\Hom_{\CC}(T,-)~: \CC/(\add T[1]) \xrightarrow{\sim} \modd B_T.$$	
	\end{theorem}

	The equivalence $\Hom_{\CC}(T,-)$ induces an equivalence between $\mathcal U/(\add T[1])$ and the subcategory of mod-$B_T$ consisting of $B_T$-modules supported on the unfrozen part of $(\mathcal Q,F)$, which is denoted by $Q$. Slightly abusing notations, an object $M$ in $\mathcal U/(\add T[1])$ is identified with its image $\Hom_{\CC}(T,M)$ in mod-$B_T$. Conversely, any $B_T$-module supported on $Q$ is viewed as an object in $\mathcal U/(\add T[1])$.

	For any $i \in Q_0$, we denote by $S_i$ the simple $B_T$-module corresponding to the point $i$. We denote by $\<-,-\>$ the \emph{truncated Euler form} on mod-$B_T$ defined by
	$$\<M,N\>=\dim \Hom_{B_T}(M,N) - \dim \Ext^1_{B_T}(M,N)$$
	for any $B_T$-modules $M$ and $N$. We denote by $\<-,-\>_a$ the \emph{anti-symmetrised Euler form} on mod-$B_T$ defined by
	$$\<M,N\>_a=\<M,N\>-\<N,M\>$$
	for any $B_T$-modules $M$ and $N$.
	
	\begin{lem}[\cite{Palu}]\label{lem:SiK0}
		For any $i \in \mathcal Q_0$, the form $M \mapsto \<S_i,M\>$ on mod-$B_T$ only depends on the class $[M]$ of $M$ in the Grothendieck group $K_0(\modd B_T)$ of mod-$B_T$.
	\end{lem}

	For any $B_T$-module $M$ and any $\mathbf e \in K_0(\modd B_T)$, we let $\Gr_{\mathbf e}(M)$ denote the set of submodules $N$ of $M$ whose class $[N]$ in $K_0(\modd B_T)$ equals $\mathbf e$. This set is called the \emph{Grassmannian of submodules of $M$ of dimension $\mathbf e$}. It is a projective variety and we denote by $\chi(\Gr_{\mathbf e}(M))$ its Euler-Poincar\'e characteristic (with respect to the singular cohomology if $\k$ is the field of complex numbers, and to the \'etale cohomology with compact support if $\k$ is arbitrary).

	\begin{defi}[\cite{Palu}]
		The \emph{cluster character} associated with $(\CC,T)$ is the unique map
		$$X^T_?: \Ob(\CC) \fl \mathcal L(\mathbf x_{\mathcal Q})$$
		such that
		\begin{enumerate}[a)]
			\item $X^T_{T_i[1]}=x_i$ for any $i \in \mathcal Q_0$~;
			\item If $M$ is indecomposable and not isomorphic to any $T_i[1]$, then
				$$X^T_M=\sum_{\mathbf e \in \N^{\mathcal Q_0}}\chi(\Gr_{\mathbf e}(\Hom_{\CC}(T,M)))\prod_{i \in \mathcal Q_0}x_i^{\<S_i,\mathbf e\>_a-\<S_i,\Hom_{\CC}(T,M)\>}~;$$
			\item For any two objects $M,N$ in $\CC$,
				$$X^T_{M \oplus N}=X^T_{M}X^T_{N}.$$
		\end{enumerate}
	\end{defi}

	We recall that an indecomposable object $X$ in $\CC$ is called \emph{reachable from $T$} if it is a direct summand of a tilting object which can be obtained from $T$ by a finite number of mutations, see \cite{BIRS}. In particular, any reachable object is rigid.

	\begin{theorem}[\cite{FK}]\label{theorem:FK}
		The map $X^T_?$ induces a surjection from the set of indecomposable objects in $\CC$ which are reachable from $T$ to the set of cluster variables in the cluster algebra $\mathcal A(\mathcal Q,F)$.
	\end{theorem}

		Note that identifying $\CC/(\add T[1])$ with $\modd B_T$ using the functor $\Hom_{\CC}(T,-)$, we associate to any indecomposable $B_T$-module $M$ the cluster character of an indecomposable lifting $\overline M$ of $M$ in $\CC/(\add T[1])$. Thus, we set
		$$X^T_M = X^T_{\overline M}=\sum_{\mathbf e \in \N^{\mathcal Q_0}}\chi(\Gr_{\mathbf e}(M))\prod_{i \in \mathcal Q_0}x_i^{\<S_i,\mathbf e\>_a-\<S_i,M\>}$$
		and this way, we may view $X^T_?$ as a map on the set of objects in mod-$B_T$.

% 		In general, it is complicated to compute explicitly the cluster character associated with a $B_T$-module $M$. If the module $M$ is a string module, Cerulli and Haupt provided methods for computing $\chi(\Gr_{\mathbf e}(M))$ \cite{Cerulli:string, Haupt:string} but the exponents in the cluster character formula still remain to be computed. In the present article, we prove that the formula $M \mapsto L_M$ defined in Section \ref{section:formula} completely solves this question for string $B_T$-modules.

	\section{A formula for Dynkin type $\A$ with coefficients}\label{section:typeA}
		We now start the proof of our main result, which will be stated in Theorem \ref{theorem:main}. This section is devoted to the first step of the proof in which we establish the theorem for specific modules in the particular case of blown-up quivers with unfrozen part of Dynkin type $\A$.
			The result we prove in this section is the following~:
		\begin{theorem}\label{theorem:typeA}
			Let $(\mathcal Q,F)$ be a blown-up ice quiver with unfrozen part $Q$ of Dynkin type $\A$ and let $(\mathcal Q,F,\CC,T)$ be a realisable quadruple. Let $B_T=\End_{\mathcal{C}}(T)$. Then for any indecomposable $B_T$-module $M$, the following hold~:
			\begin{enumerate}[a)]
				\item If $M$ is unfrozen, then it is a string module~;
				\item If $M$ is a submodule of the unique unfrozen sincere module, then the corresponding cluster character is given by
					$$X^T_M = L_{M}.$$
			\end{enumerate}
		\end{theorem}

		Every $B_T$-module $M$ supported on $Q$ has the structure of a module over the path algebra $H=\k Q$ of the quiver $Q$. Since $Q$ is of Dynkin type $\A$ every such module $M$ is a string module if it is indecomposable. This proves the first point of the theorem.

		To prove the second point, we need to collect the necessary background concerning the matrix product formula from \cite{ARS:frises}, which we do in Section \ref{ssection:ARS}. Section \ref{ssection:typeAstep1} is devoted to the proof when $(\mathcal Q,F)$ is the principal extension of a Dynkin quiver of type $\A$. In Section \ref{ssection:typeAstep2}, we use the Fomin-Zelevinsky separation formula in order to deduce the general case.		

		\subsection{Background on the matrix product formula}\label{ssection:ARS}
			Let $Q$ be a Dynkin quiver of type $\A_n$ with $n \geq 1$. Attach to each point $i$ of $Q_0$ a cluster variable $x_i$. Because $Q$ is of Dynkin type $\A_n$, it is well-known that the corresponding coefficient-free cluster algebra $\mathcal A(Q,\emptyset)$ is generated by $\frac{n(n+1)}{2}+n$ cluster variables. If $\CC_Q$ denotes the cluster category of $Q$, the set of indecomposable objects in $\CC_Q$ can be identified with the disjoint union of the set of indecomposable $\k Q$-modules and $\ens{P_i[1] \ | \ i\in Q_0}$ where $P_i$ denotes the indecomposable $\k Q$-module associated with the point $i \in Q_0$. Identifying the path algebra $\k Q$ with a tilting object in $\CC_Q$, the quadruple $(Q,\emptyset,\CC_Q,\k Q)$ is realisable and the associated cluster character $M \mapsto X_M = X^{\k Q}_M$ induces a bijection from the set of indecomposable objects in $\mathcal{C}_{Q}$ to the set of cluster variables in $\mathcal A(Q,\emptyset)$. Moreover, this bijection sends each object of the form $P_i[1]$ onto $x_i$, and each indecomposable $\k Q$-module $M$ onto the unique cluster variable having, in its reduced form, $\mathbf x^{\ddim M}$ as denominator, where $\ddim M = (\mathrm{dim}_{\k} M(i)) \in \N^{Q_0}$, see \cite{CK2}.

			Under this identification, one can position the cluster variables of $\mathcal A(Q,\emptyset)$ into a grid underlying the Auslander-Reiten quiver of $\mathcal{C}_Q$. This positioning of variables on the grid corresponds to the frieze on the repetition quiver $\Z Q^{\op}$ associated with the opposite quiver of $Q$, as constructed in \cite{ARS:frises}. We illustrate this on an example.

			Let $Q$ be a Dynkin quiver of type $\mathbb{A}_n$, say
			\[
			\xymatrix@!C=5pt{1 \ar[r]^-{\alpha} & 2 \ar[r]^-{\beta} & 3 & \ar[l]_-{\gamma} 4 \ar[r]^-{\delta} & 5 & \ar[l]_-{\epsilon} 6 & \ar[l]_-{\zeta} 7  & \ar[l]_-{\eta} 8 \ar[r]^-{\theta} & 9  \ar[r]^-{\iota} & 10  & \ar[l]_-{\kappa} 11.
			}
			\]
			(here $n=11$).  The corresponding grid, in which we illustrated the cluster variables associated with the indecomposable projective $\k Q$-module $P_i$, and those associated with the indecomposable injective $\k Q$-module $I_i$, is as follows (the variable $X_{M(u,v)}$ will be explained thereafter).
			$$
			\begin{array}{cccccccccccccc}
			&&&&x_{10}&x_{11} \\
			&&&&x_{9} &X_{P_{10}}& X_{P_{11}} \\
			&x_5 & x_6 & x_7 & x_8 &X_{P_9}&*&\ddots\\
			x_3 & x_4 &X_{P_5}&X_{P_6}&X_{P_7}&X_{P_8}&*&&\ddots\\
			x_2 &X_{P_3}&X_{P_4}&X_{M_{(u,v)}}&&&*&&&\ddots\\
			x_1 &X_{P_2}&&&&&*&&&&\ddots\\
				&X_{P_1}&*&*&*&*&*&*&*&X_{I_3}&X_{I_2}&X_{I_1} \\
			&&\ddots&&&&*&&X_{I_5}&X_{I_4}&x_3&x_2&x_1\\
			&&&\ddots&&&*&&X_{I_6}&x_5&x_4\\
			&&&&\ddots&&*&&X_{I_7}&x_6 \\
			&&&&&\ddots&X_{I_{10}}&X_{I_9}&X_{I_8}&x_7 \\
			&&&&&&X_{I_{11}}&x_{10}&x_{9}&x_{8}\\
			&&&&&&&x_{11}\\
			%&&&&&&&&x_{12}
			\end{array}
			$$

			In \cite[\S 8.2]{ARS:frises}, the authors defined, for each cell $(u,v)$ in the grid, a Laurent polynomial $t_{Q^{\op}}(u,v)$, whose definition depends on the region of the grid in which the cell is located.

			For the purpose of our paper, it suffices to consider only the north-west component of the grid together with its south-east frontier of asterisks. It is important to observe that in mod-$\k Q$, this part of the grid contains exactly all the indecomposable submodules of the unique indecomposable sincere $\k Q$-module, which is exactly positioned at the intersection of the vertical and the horizontal lines of asterisks (see for instance \cite{Gabriel:AR}). This observation will be crucial in Lemma \ref{lem:ARS} below and in Section \ref{section:main}.

			We give more details on this region. For any cell $(u,v)$ located in the north-west component of the grid (and not on its south-east frontier of asterisks), its horizontal and vertical projections onto the initial variables $x_i$ determine a word $x_{k}x_{k+1}\dots x_{k+l+1}$.  Following \cite{ARS:frises}, for each $j=1, 2, \dots, l-1$, let
			\[
			M(x_{j}, x_{{j+1}})=
			\left\{
			\begin{array}{ll}
			\left[\begin{array}{cc} x_{j} & 1\\ 0 & x_{{j+1}}\end{array}\right] & \text{if $x_{j}$ is to the left of $x_{{j+1}}$}, \\ [3mm]
			\left[\begin{array}{cc} x_{{j+1}} & 0\\ 1 & x_{{j}}\end{array}\right] & \text{if $x_{j}$ is below $x_{{j+1}}$}.
			\end{array}\right.
			\]
			Then
			\[
			t_{Q^{\op}}(u,v)=\frac{1}{x_{k+1}\cdots x_{k+l}}[1,x_{k}]\left(\displaystyle\prod_{j=1}^{l-1}M(x_{k+j}, x_{k+{j+1}})\right)\left[\begin{array}{c}1 \\ x_{k+{l+1}}\end{array}\right].
			\]
			To determine the corresponding indecomposable $\k Q$-module $M_{(u,v)}$ for which $t_{Q^{\op}}(u,v)=X_{M(u,v)}$, the horizontal and vertical projections onto the cluster variables $X_{P_i}$ gives the word  $X_{P_{k+1}}X_{P_{k+2}}\dots X_{P_{k+l}}$, meaning that $M_{(u,v)}$ is the string module corresponding to the unique string from the point $k+1$ to the point $k+l$ in $Q$.
			In our example, the cell $(u,v)$ gives rise to the word $x_2x_3\dots x_7$, leading to
			\[
			t_{Q^{\op}}(u,v)=\frac{1}{x_{3}x_{4}x_{5}x_{6}}
			[1,x_{2}]
			\left[\begin{array}{cc} x_3 & 1 \\ 0 & x_4\end{array}\right]
			\left[\begin{array}{cc} x_5 & 0 \\ 1 & x_4\end{array}\right]
			\left[\begin{array}{cc} x_5 & 1 \\ 0 & x_6\end{array}\right]
			\left[\begin{array}{c} 1 \\ x_7\end{array}\right]
			\]
			and $M_{(u,v)}$ corresponds to the string module whose corresponding string is $\gamma^{-1}\delta\epsilon^{-1}$ in $Q$.

			Now, a close inspection of the formulae (in the coefficient-free situation) gives $t_{Q^{\op}}(u,v)=L_{M_{(u,v)}}$ whenever $(u,v)$ lies in the north-west region.
			
			As mentioned above, we also need to consider the south-east frontier of asterisks of the north-west region. Observe that on this frontier, $L_{M_{(u,v)}}$ is defined, while $t_{Q^{\op}}(u,v)$ is not. To fix this, one can augment ${Q^{\op}}$ with two new sinks, labeled $0$ and $n+1$, in order to obtain a Dynkin quiver $\overline{{Q^{\op}}}$ of type $\A_{n+2}$ in such a way the cells which were on the frontier of asterisks now lie in the north-west region in the grid corresponding to $\overline{{Q^{\op}}}$; thus $t_{\overline{{Q^{\op}}}}(u,v)$ can be defined. In our example, it suffices to let $\overline{{Q^{\op}}}$ be given by
			\[
			\xymatrix@!C=5pt{0&\ar[l]1&\ar[l]_-{\alpha} 2 & \ar[l]_-{\beta} 3 \ar[r]^-{\gamma} & 4 & \ar[l]_-{\delta} 5 \ar[r]^-{\epsilon} &6 \ar[r]^-{\zeta} & 7  \ar[r]^-{\eta} &  8 & \ar[l]_-{\theta} 9  & \ar[l]_-{\iota} 10  \ar[r]^-{\kappa} & 11\ar[r]&12.
			}
			\]
			It is then easily checked that $L_{M_{(u,v)}}=t_{\overline{{Q^{\op}}}}(u,v)|_{\substack{x_0=1\\x_{n+1}=1}}$, that is, $L_{M_{(u,v)}}$ is obtained from $t_{\overline{{Q^{\op}}}}(u,v)$ by specialising the initial cluster variables $x_0$ and $x_{n+1}$ to $1$. Observe that this relation also holds true for any cell located in the north-west region.  So, in general, one can write
			\[
			L_{M_{(u,v)}}=t_{\overline{{Q^{\op}}}}(u,v)|_{\substack{x_0=1\\x_{n+1}=1}}
			\]
			whenever $(u,v)$ lies in the north-west region or on its south-east frontier of asterisks.
					
%			For instance, in our example, we have $V_{\s(M)}(4)=V_{\s(M)}(5)=I_2$, and thus
			%	\[
			%L_{M_{(u,v)}}=\frac{1}{x_{3}x_{4}x_{5}x_{6}}
			%[1,1]
			%\left[\begin{array}{cc} 1 & 0 \\ 0 & x_2\end{array}\right]
			%\left[\begin{array}{cc} x_3 & 1 \\ 0 & x_4\end{array}\right]
			%\left[\begin{array}{cc} x_5 & 0 \\ 1 & x_4\end{array}\right]
			%\left[\begin{array}{cc} x_5 & 1 \\ 0 & x_6\end{array}\right]
			%\left[\begin{array}{cc} 1 & 0 \\ 0 & x_7\end{array}\right]
			%\left[\begin{array}{c} 1 \\ 1\end{array}\right]
			%\]	
					
			In the situation where we deal with arbitrary coefficients, let $(\mathcal{Q}, F)$ be an ice quiver with unfrozen part $Q$ of Dynkin type $\mathbb{A}$. Then, generalising the coefficient-free situation, we let
			\[
			t_{{Q^{\op}}}(u,v)=L_{M_{(u,v)}}.
			\]
			For each $i\in Q_0$, let
			\[
			y_i=\displaystyle\prod_{\alpha \in \mathcal Q_1(F,i)}x_{s(\alpha)} \quad \text{and} \quad z_i=\displaystyle\prod_{\alpha \in \mathcal Q_1(i,F)}x_{t(\alpha)}.
			\]
			and for any $\mathbf d =(d_i)_{i \in Q_0} \in \N^{Q_0}$, we set $\y^{\mathbf d} = \prod_{i \in Q_0} y_i^{d_i}$ and $\z^{\mathbf d} = \prod_{i \in Q_0} z_i^{d_i}$. 

			Then, a tedious and combinatorial adaptation of Lemmata 5, 6 and Theorem 4 in \cite{ARS:frises}, in which one needs to embed ${Q^{\op}}$ in $\overline{{Q^{\op}}}$ as above, allows us to obtain the following recurrence relations, whose verification is left to the reader.
			\begin{lem}\label{lem:ARS}
				\begin{enumerate}[a)]
					\item For any $i\in Q_0$, we have
					\begin{equation}\label{eq:L_xi}
					x_i L_{P_i} - y_i \left(\prod_{\alpha \in Q_1(i,-)} L_{P_{t(\alpha)}}\right) \left( \prod_{\alpha \in Q_1(-,i)} x_{s(\alpha)} \right) = \mathbf z^{\ddim P_i}.
					\end{equation}
					\item For any non-projective indecomposable submodule $M$ of the unique indecomposable unfrozen sincere $\k Q$-module, we have
					\begin{equation}
					L_{\tau M} L_M - L_E = \mathbf y^{\ddim \tau M} \mathbf z^{\ddim M}
					\end{equation}
					where $E$ is the middle term of the almost split exact sequence $0 \fl \tau M \xrightarrow{i} E \xrightarrow{p} M \fl 0$ in \emph{mod-}$\k Q$.
				\end{enumerate}
			\end{lem}

	\subsection{Proof of Theorem \ref{theorem:typeA} - first step : Principal coefficients}\label{ssection:typeAstep1}
		We start with two lemmata concerning cluster characters associated with principal extensions of acyclic quivers. These are analogues to \cite[Lemma 3.9 and Proposition 3.10]{CC} (see also \cite[Lemma 2.2]{Dupont:qChebyshev}) which provide representation-theoretical interpretations of certain exchange relations in a cluster algebra with principal coefficients at an acyclic seed.

		Let $Q$ be an acyclic quiver and let $H=\k Q$ be the path algebra of $Q$, which is finite-dimensional and hereditary. The map sending an $H$-module module to its dimension vector allows us to identify the Grothendieck group $K_0(\modd H)$ with $\Z^{Q_0}$. If $(\mathcal Q,F)$ is the principal extension $(Q^\pp, Q_0')$ of $Q$ defined in Example \ref{exmp:Qpp}, we denote by $X^\pp_?$ the corresponding cluster character and by $L^\pp$ the matrix formula of equation \eqref{eq:Lc}. For every $i' \in Q_0'$, we set $y_i = x_{i'}$. It is known that
		$$X^\pp_M = \sum_{\mathbf e \in \N^{Q_0}} \chi(\Gr_{\mathbf e}(M)) \prod_{i \in Q_0} x_i^{-\<\mathbf e,[S_i]\>-\<[S_i],[M]-\mathbf e\>} y_i^{m_i - e_i}$$
		where the Euler forms and the Grassmannian are considered in mod-$H$ and where $[M]=(m_i)_{i \in Q_0}$ (see for instance \cite[Remark 2.4]{Dupont:qChebyshev}).
		\begin{lem}\label{lem:ADSS x_i}
			Let $Q$ be an acyclic quiver. Then for any $i \in Q_0$, we have
			\begin{equation}\label{eq:proj}
				x_i X^\pp_{P_i} - y_i \left(\prod_{\alpha \in Q_1(i,-)} X^\pp_{P_{t(\alpha)}}\right) \left( \prod_{\alpha \in Q_1(-,i)} x_{s(\alpha)} \right) = 1.
			\end{equation}
			where $P_j$ is the indecomposable projective $\k Q$-module associated with the point $j\in Q_0$.
		\end{lem}
		\begin{proof}
			The proof is a straightforward adaptation of the proof of \cite[Lemma 3.9]{CC}, we give it for completeness. We recall that for any $i \in Q_0$, we have
			$$\rad P_i = \bigoplus_{\alpha \in Q_1(i,-)} P_{t(\alpha)} \textrm{ and thus }X^{\pp}_{\rad P_i} = \prod_{\alpha \in Q_1(i,-)} X^\pp_{P_{t(\alpha)}}.$$
			We also recall that $P_i / \rad P_i \simeq S_i$ and that a submodule $M$ of $P_i$ either equals $P_i$ or is a submodule of $\rad P_i$. We set $\ddim P_i = \mathbf m = (m_j)_{j \in Q_0}$ and let $\delta$ be such that $\delta_{ij}=1$ if $i=j$ and $\delta_{ij}=0$ otherwise. Thus
			\begin{align*}
				X^\pp_{\rad P_i}
					& = \sum_{\mathbf e} \chi(\Gr_{\mathbf e}(\rad P_i)) \prod_{l \in Q_0} x_l^{-\<\mathbf e, [S_l]\>-\<[S_l], \mathbf m - [S_i] - \mathbf e\>} y_l^{m_l - \delta_{il} - e_l} \\
				& = \sum_{\mathbf e} \chi(\Gr_{\mathbf e}(\rad P_i)) \prod_{l \in Q_0} \left(x_l^{-\<\mathbf e, [S_l]\>-\<[S_l], \mathbf m - \mathbf e\>} y_l^{m_l - e_l}\right) x_l^{\<[S_l], [S_i]\>} y_l^{-\delta_{il}}\\
				& = y_i^{-1} \left( \prod_{\alpha \in Q_1(-,i)} x_{s(\alpha)}^{-1} \right) x_i \sum_{\mathbf e} \chi(\Gr_{\mathbf e}(\rad P_i)) \prod_{l \in Q_0} x_l^{-\<\mathbf e, [S_l]\>-\<[S_l], \mathbf m - \mathbf e\>} y_l^{m_l - e_l} \\
			\end{align*}
			but
			$$X^\pp_{P_i} = x_i^{-1} + \sum_{\mathbf e} \chi(\Gr_{\mathbf e}(\rad P_i)) \prod_{l \in Q_0} x_l^{-\<\mathbf e, [S_l]\>-\<[S_l], \mathbf m - \mathbf e\>} y_l^{m_l - e_l}.$$
			Thus,
			$$ x_i X^\pp_{P_i} = y_i \left( \prod_{\alpha \in Q_1(-,i)} x_{s(\alpha)} \right) X^\pp_{\rad P_i} + 1$$
			from which we deduce \eqref{eq:proj}.
		\end{proof}

		\begin{lem}\label{lem:ADSS M}
			Let $Q$ be an acyclic quiver. Then for any non-projective $\k Q$-module $M$, we have
			\begin{equation}\label{eq:nonproj}
				X^\pp_{\tau M} X^\pp_M - X^\pp_E = \mathbf y^{\ddim \tau M}
			\end{equation}
			where $E$ is the central term of the almost split exact sequence $0 \fl \tau M \xrightarrow{i} E \xrightarrow{p} M \fl 0$.
		\end{lem}
		\begin{proof}
			This result is an analogue of \cite[Proposition 3.10]{CC}. We sketch the proof for the convenience of the reader.
		
			We write $N = \tau M$, $\ddim N = \mathbf n = (n_l)_{l \in Q_0}$ and $\ddim M = \mathbf m = (m_l)_{l \in Q_0}$. It follows from the definition of the character that we have
			$$X^\pp_{N} X^\pp_M = X^\pp_{N \oplus M} = \sum_{\mathbf e \in \N^{Q_0}} \chi(\Gr_{\mathbf e}(N \oplus M)) \prod_{l \in Q_0} x_l^{-\<\mathbf e, [S_l]\>-\<[S_l], \mathbf n +\mathbf m - \mathbf e\>} y_l^{m_l + n_l - e_l.}$$
			but there is a surjective map with affine fibres
			$$\Gr_{\mathbf e}(N \oplus M) \fl \bigsqcup_{\mathbf f + \mathbf g =\mathbf e} \Gr_{\mathbf f}(N) \times \Gr_{\mathbf g}(M)$$
			so that $\chi(\Gr_{\mathbf e}(N \oplus M)) = \sum_{\mathbf f + \mathbf g =\mathbf e} \chi(\Gr_{\mathbf f}(N))\chi(\Gr_{\mathbf g}(M))$ and thus
			$$X^\pp_{N \oplus M} = \sum_{\mathbf f, \mathbf g} \chi(\Gr_{\mathbf f}(N)) \chi(\Gr_{\mathbf g}(M)) \prod_{l \in Q_0} x_l^{-\<\mathbf f + \mathbf g, [S_l]\>-\<[S_l], \mathbf n +\mathbf m - \mathbf f - \mathbf g\>} y_l^{m_l + n_l - f_l -g_l.}$$
			Now, it follows from \cite[Lemma 3.11]{CC} that every fibre of the map
			$$\zeta : \left\{\begin{array}{rcl}
				\Gr_{\mathbf e}(N \oplus M) & \fl & \displaystyle \bigsqcup_{\mathbf f + \mathbf g = \mathbf e} \Gr_{\mathbf f}(N) \times \Gr_{\mathbf g}(M) \\
				U & \mapsto & (i^{-1}(U), p(U))
			\end{array}\right.$$
			is an affine space except over the point $(0,M)$ where it is empty. It thus follows that
			$$X^\pp_{N \oplus M} = X^\pp_E + \prod_{l \in Q_0}y_l^{n_l}$$
			which establishes \eqref{eq:nonproj}.	
		\end{proof}

		We now prove the second point of Theorem \ref{theorem:typeA} for principal coefficients~:
		\begin{prop}\label{prop:typeApp}
			Let $Q$ be a Dynkin quiver of type $\A$. Then for any indecomposable submodule $M$ of the unique unfrozen sincere $\k Q$-module, we have
			$$X^\pp_M = L^\pp_M.$$
		\end{prop}
		\begin{proof}
			The proof directly follows from Lemmata \ref{lem:ARS}, \ref{lem:ADSS x_i} and \ref{lem:ADSS M}, keeping in mind that for principal coefficients we have $z_i=1$ for each $i$ in Lemma \ref{lem:ARS}.
			\end{proof}

	\subsection{Proof of Theorem \ref{theorem:typeA} - second step}\label{ssection:typeAstep2}
		We now finish the proof of Theorem \ref{theorem:typeA}. For this, the strategy is to apply the Fomin-Zelevinsky separation formula to the equality established in Proposition \ref{prop:typeApp}. We let $(\mathcal Q,F)$ be a blown-up ice quiver with unfrozen part $Q$ of Dynkin type $\A$ and we fix a realisable quadruple $(\mathcal Q,F,\CC,T)$ and an unfrozen indecomposable $B_T$-module $M$ which has a natural structure of $H$-module where $H$ is the path algebra of $Q$. We denote by $(Q^\pp,Q_0')$ the principal extension of $Q$. Every $H$-module can naturally be viewed as a $B_T$-module or as a $\k Q^{\pp}$-module. According to Proposition \ref{prop:typeApp}, we know that $L^\pp_M = X^\pp_{M}$. We now want to prove that $L_M = X^T_M$ where $L_?$ denotes the matrix formula associated with $(\mathcal Q,F)$ in equation \eqref{eq:Lc}.

		In order to simplify the notations, we identify $Q_0$ with $\ens{1, \ldots, n}$ and for every $i \in Q_0$, we denote $i'$ by $n+i$. Let $\mathbb P$ be the tropical semifield generated by the $x_i$, with $i \in F$, and endowed with the auxiliary addition
		\[
		\prod_{i}x_i^{a_i} \oplus \prod_{i}x_i^{b_i} = \prod_{i}x_i^{\text{min}\{a_i, b_i\}}.
		\]
		For any $i \in Q_0$, we set
		$$w_i = \left( \prod_{\alpha \in \mathcal Q_1(F,i)} x_{s(\alpha)} \right)\left( \prod_{\alpha \in \mathcal Q_1(i,F)} x_{t(\alpha)}^{-1} \right).$$
		With the notations of Section \ref{ssection:ARS}, we can thus write $w_i=y_iz_i^{-1}$.
		Following \cite{cluster4}, for every subtraction-free rational expression $f$ in the variables $x_1, \ldots, x_{2n}$, we define the \emph{separation of $f$} as
		$$\sigma(f) = \frac{f(x_1, \ldots, x_n, w_1, \ldots,w_n)}{f|_{\mathbb P}(1, \ldots, 1, w_1, \ldots,w_n)}$$
		where the $f|_{\P}$ means that we have replaced the ordinary addition in $\mathcal F(\x_{F})$ by the auxiliary addition $\oplus$ of the semifield $\P$. This can be done since $f$ is subtraction-free.

		Now, we note that, for every $H$-module, $L^\pp_M$ is a subtraction-free rational expression by definition. Also, for every indecomposable $H$-module $M$, $X^\pp_M$ is a cluster variable so that it is defined with a finite number of mutations, which are all subtraction-free. Thus, $X^\pp_M$ is also a subtraction-free rational expression (see \cite[\S 3]{cluster4}) and we can apply $\sigma$ to both $X^\pp_M$ and $L^\pp_M$. Since $L^\pp_M = X^\pp_{M}$ by Proposition \ref{prop:typeApp}, we have $\sigma(L^\pp_M) = \sigma(X^\pp_M)$. Thus, we only need to prove that $\sigma(X^\pp_M) = X^T_M$ and $\sigma(L^\pp_{M}) = L_M$.

		The equality $\sigma(X^\pp_M) = X^T_M$ follows directly from \cite[Theorem 3.7]{cluster4} since $X^\pp_?$ and $X^T_?$ induce bijections from the set of indecomposable $H$-modules to the set of cluster variables in the cluster algebras $\mathcal A(Q^\pp,Q_0')$ and $\mathcal A(\mathcal Q,F)$ respectively and these bijections respect denominator vectors.

		The equality $\sigma(L^\pp_M) = L_M$ is obtained from the following observation. We write $c= \s(M)$. For each $j\in F$, $x_j$ appears in exactly one $w_i$, since there is exactly one $i \in Q_0$ which is adjacent to $j$ in $\mathcal Q$. On the other hand, in the product of matrices $L_M^\pp$, $x_{n+i}$ appears in at most one matrix, namely in the matrix $V^\pp_c(i)$ (where $V^\pp_c(i)$ denotes the matrices arising in the matrix product $L_c^{\pp}$). Thus, using the definition of the addition in $\mathbb{P}$ and the fact that $L_M^\pp(1,\ldots,1,x_{n+1},\ldots,x_{2n})$ has constant term $1$,
		%(since $L_M^\pp(x_1,\ldots,x_n,x_{n+1},\ldots,x_{2n})$ is a cluster variable in a cluster algebra of Dynkin type $\A$),
		it follows that
		$$L_M^\pp|_{\mathbb{P}} (1,\ldots,1,w_1,\ldots,w_n) =\prod_{i \in (\overline{\supp(M)})_0} \prod_{\alpha \in Q^{\pp}_1(i,F)} x_{t(\alpha)}^{-1} .$$
		Now, since
		$$ \left( \prod_{\alpha \in \mathcal Q_1(i,F)} x_{t(\alpha)} \right) \left(V^\pp_c(i)(x_1,\ldots,x_n,w_1,\ldots,w_n)\right) =V_c(i),$$
		we get $\sigma(L_M^\pp)=L_M$. This ends the proof of Theorem \ref{theorem:typeA}. \hfill \qed
		
\section{The general case}\label{section:main}
	We now deduce the main theorem from Theorem \ref{theorem:typeA}. For this, we use some ``blow-up'' techniques which are described in Section \ref{ssection:coverings}. These techniques introduce some ``error'' in the computation of the characters but this can be controlled with a normalising factor that we introduce in Section \ref{ssection:normalisation}. The main result (Theorem \ref{theorem:main}) is stated in Section \ref{ssection:main}.

	\subsection{Blowing up quivers along string modules}\label{ssection:coverings}
		In this section, we fix a quiver $Q$, we denote by $M$ a string representation of $Q$ and we write $c=\s(M)$. If $c$ is of positive length $n \geq 1$, we write $c = c_1 \cdots c_n$.

		The \emph{blow-up $\widetilde{Q_M}$ of $Q$ along $M$} is the quiver constructed as follows. Let $\ens{v_1, \ldots, v_{n+1}}$ be a set. For any $i \in \ens{1, \ldots, n}$, we set $\beta_i$ from $v_{i}$ to $v_{i+1}$ which is an arrow (or a formal inverse of an arrow, respectively) in $(\widetilde{Q_M})_1$ if $c_i$ is an arrow (or a formal inverse of an arrow, respectively) in $Q_1$. For any $i \in \ens{1, \ldots, n}$ and for any arrow $\alpha \in Q_1$ such that $\alpha \neq c_i^{\pm 1}, c_{i-1}^{\pm 1}$, if $s(\alpha) = t(c_i)$ (or $t(\alpha) = s(c_i)$, respectively), we create a new point, denoted by $t(\alpha)^{\alpha;i}$ (or $s(\alpha)^{\alpha;i}$, respectively) and an arrow $\alpha_{v_i} : v_i \fl t(\alpha)^{\alpha;i}$ (or $\alpha_{v_i} : s(\alpha)^{\alpha;i} \fl v_i$, respectively).

		\begin{exmp}\label{exmp:blowQ}
			Consider the quiver
			$$\xymatrix{
				Q : 	& 1 \ar[r]^{\alpha} & 2 \ar@<-2pt>[r]_{\gamma} \ar@<+2pt>[r]^{\epsilon} & 3 \ar[r]^{\delta} & 4
			}$$
			and let $M$ be the string representation corresponding to the walk $c = \epsilon^{-1}\gamma$, that is
			$$\xymatrix{
				M : 	& 0 \ar[r] & \k \ar@<-2pt>[r]_{[0,1]^t} \ar@<+2pt>[r]^{[1,0]^t} & \k^2 \ar[r] & 0.
			}$$
			Then, the quiver $\widetilde{Q_M}$ is
			$$\xymatrix{
				\widetilde{Q_M} & 2^{\gamma;1} \ar[r]^{\gamma_{v_1}} & v_1 \ar[d]_{\delta_{v_1}} & \ar[l]_{\beta_1} v_2 \ar[r]^{\beta_2} & v_3 \ar[d]^{\delta_{v_3}} & \ar[l]_{\epsilon_{v_3}} 2^{\epsilon;3}. \\
						& & 4^{\delta;1}  & 1^{\alpha;2} \ar[u]_{\alpha_{v_1}} & 4^{\delta;3}
			}$$
		\end{exmp}

		We recall the following definition from \cite{Haupt:string}~:
		\begin{defi}
			Let $Q$ and $S$ be two quivers. A \emph{winding of quivers} $\Phi: Q \fl S$ is a pair $\Phi = (\Phi_0,\Phi_1)$ where $\Phi_0: Q_0 \fl S_0$ and $\Phi_1 : Q_1 \fl S_1$ are such that~:
			\begin{enumerate}[a)]
				\item $\Phi$ is a morphism of quivers, that is $s \circ \Phi_1 = \Phi_0 \circ s$ and $t \circ \Phi_1 = \Phi_0 \circ t$~;
				\item If $\alpha,\alpha' \in Q_1$ with $\alpha \neq \alpha'$ and $s(\alpha) = s(\alpha')$, then $\Phi_1(\alpha) \neq \Phi_1(\alpha')$~;
				\item If $\alpha,\alpha' \in Q_1$ with $\alpha \neq \alpha'$ and $t(\alpha) = t(\alpha')$, then $\Phi_1(\alpha) \neq \Phi_1(\alpha')$~;
			\end{enumerate}
		\end{defi}

		With the above notations, the maps
		$$\Phi_0 : \left\{\begin{array}{rcll}
			v_i & \mapsto & t(c_{i-1}) & \textrm{ for any } i \in \ens{1, \ldots, n+1}\\
			v^{\alpha;i} & \mapsto & v & \textrm{ for any }v \in \overline{\supp(M)} \textrm{ and any } \alpha,i\\\
		\end{array}\right.$$
		and
		$$\Phi_1:\left\{\begin{array}{rcll}
			\beta_i & \mapsto & c_i & \textrm{ for any } i \in \ens{1, \ldots, n} \\
			\alpha_{v_i} & \mapsto & \alpha & \textrm{ for any arrow of the form } \alpha_{v_i}
		\end{array}\right.$$
		induce a winding of quivers $\Phi: \widetilde{Q_M} \fl \overline{\supp(M)}$.

		Let $\Phi_*$ be the map from the set of objects in $\rep(\widetilde{Q_M})$ to the set of objects in $\rep(\overline{\supp(M)})$ which associates to a representation $\widetilde V$ of $\widetilde{Q_M}$ the representation $V$ given by
		$$V(i)=\bigoplus_{j \in \Phi_0^{-1}(i)} \widetilde V(j) \textrm{ and } V(\alpha)=\bigoplus_{\beta \in \Phi_1^{-1}(\alpha)} \widetilde V(\beta)$$
		for any $i \in (\overline{\supp(M)})_0$ and any $\alpha \in (\overline{\supp(M)})_1$. At the level of dimension vectors, $\Phi_*$ induces a natural map $\phi: \N^{(\widetilde{Q_M})_0} \fl \N^{(\overline{\supp(M)})_0}$ .

		We define a representation $\widetilde M$ of $\widetilde{Q_M}$ by setting for any point $v \in (\widetilde{Q_M})_0$
		$$\widetilde M(v) = \left\{\begin{array}{ll}
			\k & \textrm{ if $v=v_i$ for some } i \in \ens{0, \ldots,n}, \\
			0 & \textrm{ otherwise}
		\end{array}\right.$$
		and for any arrow $\alpha \in (\widetilde{Q_M})_1$
		$$\widetilde M(\alpha) = \left\{\begin{array}{ll}
			1_{\k} & \textrm{ if $\alpha=\beta_i^{\pm 1}$ for some } i \in \ens{1, \ldots,n}, \\
			0 & \textrm{ otherwise.}
		\end{array}\right.$$

		\begin{lem}
			Let $B$ be a finite dimensional $\k$-algebra. Then for any string $B$-module $M$, we have $\Phi_{*}(\widetilde M) \simeq M$.
		\end{lem}
		\begin{proof}
			This follows from the construction.
		\end{proof}

		\begin{defi}
			If $B$ is a finite dimensional $\k$-algebra with bound quiver $(Q,I)$ and $M$ is a $B$-module, the \emph{border $\d M$ of $M$} is the set of points in the closure $\overline{\supp(M)}$ of the support of $M$ in $Q$ and which do not lie in the support $\supp(M)$ of $M$.
		\end{defi}

		Thus, with the above notations, the border $\d \widetilde M$ consists of all the points in $(\widetilde{Q_M})_0$ which do not lie in the support of $\widetilde{M}$.
		
		\begin{lem}\label{lem:back2typeA}
			Let $B$ be a finite dimensional $\k$-algebra. Then for any string $B$-module $M$, the pair $(\widetilde{Q_{M}}, \d \widetilde M)$ is a blown-up ice quiver whose unfrozen part is of Dynkin type $\mathbb A$. Moreover, the representation $\widetilde{M}$ is a sincere unfrozen string representation of $\widetilde{Q_{M}}$.
		\end{lem}
		\begin{proof}
			The first assertion follows from the construction of $\widetilde{Q_M}$ and $\widetilde M$. For the second assertion, we observe that $\widetilde{M}$ is supported on the unfrozen part of $\widetilde{Q_{M}}$ which is of Dynkin type $\A$. It is unfrozen sincere and indecomposable by construction so that it is a string representation. 			 
		\end{proof}
	
		\begin{exmp}
			Let $B$ be the path algebra of the quiver $Q$ considered in Example \ref{exmp:blowQ} and let $M$ be the string module considered in that same example. Then, the representation $\widetilde{M}$ of the quiver $\widetilde{Q_M}$ is
			$$\xymatrix{
				\widetilde{M}: & 0 \ar[r] & \k \ar[d] & \ar[l]_{1_\k} \k \ar[r]^{1_\k} & \k \ar[d] & \ar[l] 0 \\
						& & 0 & 0 \ar[u] & 0
			}$$
			so that $(\widetilde{Q_{M}}, \d \widetilde M)$ is indeed a blown-up ice quiver with unfrozen part of Dynkin type $\mathbb A_3$.
		\end{exmp}

	\subsection{Normalisation}\label{ssection:normalisation}
		In this section $M$ denotes a string module over a finite dimensional algebra $B$ with bound quiver $(Q,I)$.

		\begin{defi}\label{defi:normalisation}
			The \emph{normalising vector} of $M$ is $\mathbf n_M =(n_i)_{i \in \overline{\supp(M)}_0} \in \N^{\overline{\supp(M)}_0}$ given by
			$$n_i = \<S_i,M\> - \sum_{j \in \Phi_0^{-1}(i)}\<S_j,\widetilde M\>$$
			for any $i \in \overline{\supp(M)}_0$ where the first truncated Euler form is considered in mod-$B$ and the second truncated Euler form is considered in mod-$\k \widetilde{Q_{M}}$.

			The \emph{normalising factor of $M$} is $$\mathbf x^{\mathbf n_M} = \prod_{i \in \overline{\supp(M)}_0} x_i^{n_i}.$$
		\end{defi}

		This normalisation is actually easy to compute in several usual situations as it is explained in Section \ref{section:morenormalisation}.

		\begin{exmp}\label{exmp:normalisation}
			Consider the finite dimensional algebra $B$ whose ordinary quiver is
			$$\xymatrix{
				&& 3 \ar[ld]_{\gamma} \\
				Q: & 1 \ar[rr]_{\alpha} && 2 \ar[lu]_{\beta}
			}$$
			and whose relations are given by the vanishing of all paths of length two, that is $\alpha \beta = \beta \gamma = \gamma \alpha =0$. It is a cluster-tilted algebra of type $\A_3$. Consider the projective module $M$ associated with the point $1$. It is a string $B$-module with string $\alpha$. The associated blown-up quiver $\widetilde{Q_M}$ is $3^{\gamma;1} \fl v_1 \fl v_2 \fl 3^{\beta;2}$ and the representation $\widetilde M$ is $0 \fl \k \xrightarrow{1_{\k}} \k \fl 0$. Then, one has $\<S_1,M\> = \<S_{v_1},\widetilde M\> = 0$, $\<S_2,M\> = \<S_{v_2},\widetilde M\> = 1$ and $\<S_3,M\>=0$ whereas $\<S_{3^{\beta;2}},\widetilde M\> + \<S_{3^{\gamma;1}},\widetilde M\> =0-1 = -1$. Thus, the normalisation is $\n_M=(0,0,1)$.
		\end{exmp}

	\subsection{Blow-ups and cluster characters}
		We keep the notations of Section \ref{ssection:coverings}. As usual, we naturally identify $\mathcal L(\mathbf x_{\overline{\supp(M)}})$ to a subring of $\mathcal L(\mathbf x_{\mathcal Q})$. We consider the following surjective morphism of $\Z$-algebras~:
		$$\pi: \mathcal L(\mathbf x_{\widetilde{Q_M}}) \fl \mathcal L(\mathbf x_{\overline{\supp(M)}})$$
		defined by
		$$\pi(x_{v_i}) = x_{t(c_{i-1})}$$
		for any $1 \leq i \leq n+1$
		and
		$$\pi(x_{v^{\alpha;i}}) = x_v$$
		for any $v \in \overline{\supp(M)}$ and any $\alpha,i$.

		We now observe that for any unfrozen string module $M$ with respect to a realisable quadruple $(\mathcal Q,F,\CC,T)$, the Laurent polynomial $X^T_M$ is in the image of the function $\pi$.
		\begin{lem}
			Let $(\mathcal Q,F,\CC,T)$ be a realisable quadruple and let $M$ be an unfrozen string module. Then
			$$X^T_M \in {\mathcal L}({\bf x}_{\overline{\supp(M)}}).$$
		\end{lem}
		\begin{proof}
			We recall that unfrozen $B_T$-modules are identified via $\Hom_{\CC}(T,-)$ with the objects $M$ in $\mathcal U/(\add T[1])$. We have
			$$X^T_M=\sum_{\mathbf e \in \N^{\mathcal Q_0}}\chi(\Gr_{\mathbf e}(M))\prod_{i \in \mathcal Q_0}x_i^{\<S_i,\mathbf e\>_a-\<S_i,M\>}.$$
			Any dimension vector $\mathbf e$ such that $\Gr_{\mathbf e}(M) \neq \emptyset$ is supported on $\supp(M)$. Thus, if $i \in \mathcal Q_0$ is not in $\overline{\supp(M)}$, then $\<S_i,\mathbf e\>_a=0$ and $\<S_i,M\>=0$. In particular,
			$$X^T_M=\sum_{\mathbf e \in \N^{\mathcal Q_0}}\chi(\Gr_{\mathbf e}(M))\prod_{i \in \overline{\supp(M)}}x_i^{\<S_i,\mathbf e\>_a-\<S_i,M\>} \in{\mathcal L}({\bf x}_{\overline{\supp(M)}}).$$
		\end{proof}
	
		Let $\widetilde{\CC}$ be the cluster category of the quiver $\widetilde{Q_M}$ and let $\widetilde T = \k \widetilde{Q_M}$ be the path algebra of $\widetilde{Q_M}$, which is identified with a tilting object in $\widetilde{\CC}$. We denote by $\widetilde X_?$ the corresponding cluster character with values in $\mathcal L(\mathbf x_{\widetilde{Q_M}})$.
		\begin{prop}\label{prop:piX}
			Let $(\mathcal Q,F,\CC,T)$ be a realisable quadruple and let $M$ be an unfrozen string module. Then,
			$$\pi(\widetilde X_{\widetilde M})=\mathbf x^{\mathbf n_M} X^T_M.$$
		\end{prop}
		\begin{proof}
			We have
			\begin{align*}
				X^T_M
					& = \sum_{\mathbf e \in \N^{\mathcal Q_0}} \chi(\Gr_{\mathbf e}(M)) \prod_{i \in \mathcal Q_0} x_i^{\<S_i,\mathbf e\>_a-\<S_i,M\>} \\
					& = \sum_{\mathbf e \in \N^{\supp(M)_0}} \chi(\Gr_{\mathbf e}(M)) \prod_{i \in (\overline{\supp(M)})_0} x_i^{\<S_i,\mathbf e\>_a-\<S_i,M\>} \\
			\end{align*}
			Now, for any $\mathbf e \in \N^{\supp(M)_0}$, since $\Phi$ is a winding of quivers and $\Phi_*(\widetilde M) = M$, it follows from \cite[Theorem 1.2 (a)]{Haupt:string} that
			$$\chi(\Gr_{\mathbf e}(M))=\sum_{\mathbf f \in \phi^{-1}(\mathbf e)} \chi(\Gr_{\mathbf f}(\widetilde M)).$$

			Fix $\mathbf e \in \N^{\supp(M)_0}$ and $\mathbf f \in \phi^{-1}(\mathbf e)$. We now prove that for any $i \in \overline{\supp(M)}$, we have
			\begin{equation}\label{eq:Sia}
				x_{i}^{\<S_{i},\mathbf e\>_a}=\pi\left( \prod_{j \in \Phi_0^{-1}(i)} x_j^{\<S_j,\mathbf f\>_a}\right).	 
			\end{equation}
			where the anti-symmetrised Euler form in the left-hand side is taken in mod-$\k\widetilde{Q_M}$ and the anti-symmetrised Euler forms in the right-hand side are taken in mod-$B$.

			First note that
			$$\pi\left( \prod_{j \in \Phi_0^{-1}(i)} x_j^{\<S_j,\mathbf f\>_a}\right)=x_i^{\sum_{j \in \Phi_0^{-1}(i)} \<S_j,\mathbf f\>_a}$$
			so that it is enough to prove that
			$$\<S_{i},\mathbf e\>_a=\sum_{j \in \Phi_0^{-1}(i)} \<S_j,\mathbf f\>_a.$$
		
			For any module $U$ and any integer $n \geq 1$, we denote by $nU$ the direct sum of $n$ copies of $U$. Now, since $\<S_i,-\>_a$ is well-defined on the Grothendieck group $K_0(\modd B_T)$ (Lemma \ref{lem:SiK0}), we have
			\begin{align*}
				\<S_i,\mathbf e \>_a
					& = \sum_{k \in \supp(M)_0} \<S_i,e_k S_k\>_a \\
					& = \sum_{k \in \supp(M)_0} \<S_i,\left(\sum_{l \in \Phi_0^{-1}(k)} f_l \right) S_k\>_a \\
					& = \sum_{k \in \supp(M)_0} \left(\sum_{l \in \Phi_0^{-1}(k)} f_l \right) \<S_i, S_k\>_a
			\end{align*}
			and
			\begin{align*}
				\sum_{j \in \Phi_0^{-1}(i)} \<S_j,\mathbf f \>_a
					& = \sum_{k \in \supp(M)_0} \sum_{l \in \Phi_0^{-1}(k)} \sum_{j \in \Phi_0^{-1}(i)} f_l \<S_j,S_l\>_a \\
					& = \sum_{k \in \supp(M)_0} \sum_{l \in \Phi_0^{-1}(k)} f_l \left( \sum_{j \in \Phi_0^{-1}(i)} \<S_j,S_l\>_a \right).
			\end{align*}
			But for any $k \in \supp(M)_0$ and any $l \in \Phi_0^{-1}(k)$, we have
			$$\sum_{j \in \Phi_0^{-1}(i)} \<S_j,S_l\>_a = \<S_i,S_k\>_a$$
			which proves \eqref{eq:Sia}.

			Now,
			$$\widetilde X_{\widetilde M} = \sum_{\mathbf f \in \N^{(\widetilde{Q_M})_0}} \chi(\Gr_{\mathbf f}(\widetilde M)) \prod_{i \in (\widetilde{Q_M})_0} x_i^{\<S_i,\mathbf f \>_a - \<S_i,\widetilde M\> }$$
			so that,
			\begin{align*}
				\pi(\widetilde X_{\widetilde M})
					& = \sum_{\mathbf e \in \N^{\overline{\supp(M)}_0}} \sum_{\mathbf f \in \phi^{-1}(\mathbf e)} \chi(\Gr_{\mathbf f}(\widetilde M)) \prod_{i \in \overline{\supp(M)}_0} \pi\left(\prod_{j \in \Phi_0^{-1}(i)} x_j^{\<S_j,\mathbf f \>_a - \<S_j,\widetilde M\>} \right)\\
					& = \sum_{\mathbf e \in \N^{\overline{\supp(M)}_0}} \left(\sum_{\mathbf f \in \phi^{-1}(\mathbf e)} \chi(\Gr_{\mathbf f}(\widetilde M)) \right) \prod_{i \in \overline{\supp(M)}_0} x_i^{\<S_i,\mathbf e \>_a} \pi\left(\prod_{j \in \Phi_0^{-1}(i)} x_j^{- \<S_j,\widetilde M\>}\right)\\
					& = \sum_{\mathbf e \in \N^{\overline{\supp(M)}_0}} \chi(\Gr_{\mathbf e}(M)) \prod_{i \in \overline{\supp(M)}_0} x_i^{\<S_i,\mathbf e\>_a - \sum_{j \in \Phi_0^{-1}(i)} \<S_j,\widetilde M\>}\\
					& = \mathbf x^{\mathbf n_M} X^T_M.
			\end{align*}
			This finishes the proof.
		\end{proof}

	\subsection{The main theorem}\label{ssection:main}
		We can now prove the main theorem of the article~:
		\begin{theorem}\label{theorem:main}
			Let $(\mathcal Q,F,\CC,T)$ be a realisable quadruple and let $M$ be an unfrozen string module with respect to this quadruple. Then
			$$X^T_M=\frac{1}{\mathbf x^{\mathbf n_M}}L_{M}.$$
		\end{theorem}
		\begin{proof}
			We first notice that with the above notations, it follows directly from the definitions that
			$$\pi(\widetilde L_{\widetilde M})=L_{M}.$$
			Thus,
			$$X^T_M = \frac{1}{\mathbf x^{\mathbf n_M}} \pi(\widetilde X_{\widetilde M}) =  \frac{1}{\mathbf x^{\mathbf n_M}} \pi(\widetilde L_{\widetilde M}) = \frac{1}{\x^{\mathbf n_M}} L_{M}$$
			where the first equality follows from Proposition \ref{prop:piX}, the second equality from Theorem \ref{theorem:typeA} and Lemma \ref{lem:back2typeA} and the third from the above observation.
		\end{proof}

\section{Applications}\label{section:applications}
	In this section, we give several applications of Theorem \ref{theorem:main}.

	\subsection{Computing cluster variables}
		We now prove that the formula given in equation \eqref{eq:Lc} allows one to compute the cluster variables in several situations.
		\begin{corol}\label{corol:LMvars}
			Let $(\mathcal Q,F)$ be an ice quiver. Assume that there exists a realisable quadruple $(\mathcal Q,F,\CC,T)$ such that every rigid object in $\CC$ is reachable from $T$. Then for any unfrozen $\End_{\CC}(T)$-module $M$ which is rigid and indecomposable, the Laurent polynomial $\frac{1}{\x^{\mathbf n_M}} L_M$ is a cluster variable in $\mathcal A(\mathcal Q,F)$.
		\end{corol}
		\begin{proof}
			We identify $M$ with an indecomposable lifting in $\CC$ for the functor $\Hom_{\CC}(T,-)$, which is also rigid in $\CC$ by \cite{Smith:cluster, FL:cluster}. It follows from Theorem \ref{theorem:FK} that the map $M \mapsto X^T_M$ induces a surjection from the set of reachable rigid objects in  to the set of cluster variables in $\mathcal A(\mathcal Q,F)$.  It is also rigid in $\CC$, hence, by hypothesis, it is reachable and thus $X^T_{M}$ is a cluster variable in $\mathcal A(\mathcal Q,F)$.
		\end{proof}

		\begin{corol}
			Let $B_T$ be a cluster-tilted algebra. Then for any rigid string $B_T$-module $M$, the Laurent polynomial $\frac{1}{\x^{\mathbf n_M}} L_M$ is a cluster variable in the coefficient-free cluster algebra associated with the ordinary quiver of $B_T$.
		\end{corol}
		\begin{proof}
			Let $\CC$ be a cluster category, $T$ be a tilting object in $\CC$ and $Q_T$ be the ordinary quiver of $B_T$. Then the quadruple $(Q_T,\emptyset, \CC,T)$ is realisable (see Example \ref{exmp:realacyclic}). Moreover, it follows from \cite{BMRRT} that every indecomposable rigid object in $\CC$ is reachable from $T$. The result thus follows from Corollary \ref{corol:LMvars}.
		\end{proof}
	
		\begin{corol}
			Let $B_T$ be a cluster-tilted algebra of Dynkin type $\A$ or euclidean type $\widetilde{\A}$. Then, the map
			$$M \mapsto \frac{1}{\x^{\mathbf n_M}} L_M$$
			induces a bijection from the set of indecomposable rigid $B_T$-modules to the set of cluster variables in the coefficient-free cluster algebra associated with the ordinary quiver of $B_T$ which do not belong to the initial cluster.
		\end{corol}
		\begin{proof}
			Cluster-tilted algebras of Dynkin type $\A$ or euclidean type $\widetilde{\A}$ are string algebras \cite{ABCP}. Thus, each indecomposable rigid $B_T$-module is a string module. Now if $\CC$ is a cluster category and $T$ is a tilting object in $\CC$ such that $B_T = \End_{\CC}(T)$, it follows in particular from Theorem \ref{theorem:main} that for every indecomposable rigid $B_T$-module $M$, we have $X^T_M=\frac{1}{\mathbf x^{\mathbf n_M}}L_{M}$. But then it follows from \cite{Palu} that $M \mapsto \frac{1}{\x^{\mathbf n_M}} L_M$ is a bijection from the set of indecomposable rigid $B_T$-modules to the set of cluster variables in the coefficient-free cluster algebra $\mathcal A(Q_T,\emptyset)$ which do not belong to the initial cluster $\x_{Q_T}$ where $Q_T$ denotes the ordinary quiver of $B_T$.
		\end{proof}

	\subsection{Positivity in cluster algebras}\label{ssection:positivity}
		We now prove positivity results in cluster algebras using formula \eqref{eq:Lc}. We first recall the basic notions concerning positivity in cluster algebras.

		If $\P$ is an abelian group, we denote by $\Z\P$ its group ring. Let $n \geq 1$ be an integer and $\x=(x_1, \ldots, x_n)$ be an $n$-tuple of indeterminates. We denote by $\Z_{\geq 0}\P[\x^{\pm 1}]$ the semiring of subtraction-free Laurent polynomials in the variables in $\x$ with coefficients in $\Z\P$, that is, the set of elements of the form $$\frac{f(x_1, \ldots, x_n)}{x_1^{d_1} \cdots x_n^{d_n}}$$
		where $f$ is a polynomial in $n$ variables whose coefficients are non-negative linear combinations of elements of $\P$ and $d_1, \ldots, d_n \in \Z$. In particular, in the case when $\P= \ens 1$, the notation $\Z_{\geq 0}[\x^{\pm 1}]$ denotes the semiring of Laurent polynomials in variables $x_1, \ldots, x_n$ with non-negative coefficients.

		If $Q$ is a quiver without loops and 2-cycles, if $\P$ is any semifield and if $\y$ is a $Q_0$-tuple of elements of $\P$, it follows from the Laurent phenomenon that every cluster variable in the cluster algebra $\mathcal A(Q,\x_Q,\y)$ is an element of $\Z\P[\c^{\pm 1}]$ when expressed in any cluster $\c$ of $\mathcal A(Q,\x_Q,\y)$ \cite{cluster1}. The Fomin and Zelevinsky \emph{positivity conjecture}, stated in \cite{cluster1}, asserts that every cluster variable is actually an element of $\Z_{\geq 0}\P[\c^{\pm 1}]$. This conjecture was proved for rank two cluster algebras \cite{shermanz,MP:rangdeux,Dupont:positivityrank2}, cluster algebras arising from surfaces with or without punctures \cite{ST:unpunctured,S:unpunctured2,MSW:positivity}, cluster algebras with bipartite seeds \cite{Nakajima:cluster} and weaker versions were proved for acyclic cluster algebras \cite{ARS:frises,Qin,Nakajima:cluster}.

		\begin{corol}\label{corol:positivity}
			Let $(\mathcal Q,F,\CC,T)$ be a realisable quadruple and let $M$ be an unfrozen string module with respect to this quadruple. Then
			$$X^T_M \in \Z_{\geq 0}[{\x_{\overline{\supp(M)}}}^{\pm 1}].$$
		\end{corol}
		\begin{proof}
			The set $\Z_{\geq 0}[\x_{\overline{\supp(M)}}]$ is a semiring for the usual sum and product. Thus, the set $M_2(\Z_{\geq 0}[\x_{\overline{\supp(M)}}])$ is also a semiring for the usual sum and product of matrices. It follows that $L_M \in \Z_{\geq 0}[{\x_{\overline{\supp(M)}}}^{\pm 1}]$ and by Theorem \ref{theorem:main}, $X^T_M=L_M$. This proves the corollary.
		\end{proof}	
		This corollary was also obtained by Cerulli and Haupt \cite{Haupt:string,Cerulli:string}.

		\begin{corol}
			Let $B_T$ be a cluster-tilted algebra. Then $X^T_M \in \Z_{\geq 0}[{\x_{\overline{\supp(M)}}}^{\pm 1}]$ for any string $B_T$-module $M$. \hfill \qed
		\end{corol}

		We now provide a new proof of the Fomin and Zelevinsky positivity conjecture for cluster algebras arising from unpunctured surfaces. We first set some notations. If $(S,M)$ is an oriented bordered marked surface, Fomin, Shapiro and Thurston associated with any triangulation $\Gamma$ of $(S,M)$ a quiver $Q_\Gamma$ \cite{FST:surfaces}. For any semifield $\mathbb P$ and any $(Q_\Gamma)_0$-tuple $\mathbf y$ of elements of $\P$, we denote by $\mathcal A(Q_\Gamma,\x_{Q_\Gamma},\mathbf y)$ the cluster algebra with initial seed $(Q_\Gamma,\x_{Q_\Gamma},\mathbf y)$ (see \cite{cluster4}). Such a cluster algebra does not depend on the choice of the triangulation $\Gamma$ of $(S,M)$, up to an isomorphism of $\Z\P$-algebras \cite{FST:surfaces}. It is called the \emph{cluster algebra arising from the surface} $(S,M)$ and is denoted by $\mathcal A(S,M)$.

		For this class of cluster algebras, Fomin and Zelevinsky positivity conjecture amounts to saying that any cluster variable $z$ in $\mathcal A(S,M)$ belongs to $\Z_{\geq 0}\P[\x_{Q_\Gamma}^{\pm 1}]$ for any choice of triangulation $\Gamma$ \cite[\S 3]{cluster1}. This conjecture was proved in \cite{ST:unpunctured,S:unpunctured2} for unpunctured surfaces and in \cite{MSW:positivity} for arbitrary surfaces. We now provide a new independent representation-theoretical proof of the positivity conjecture for cluster algebras arising from unpunctured surfaces.

		\begin{corol}\label{corol:positivitysurface}
			Let $(S,M)$ be an unpunctured surface. Then the positivity conjecture holds for $\mathcal A(S,M)$ equipped with an arbitrary choice of coefficients.
		\end{corol}
		\begin{proof}
			Using the Fomin-Zelevinsky separation formula \cite[Theorem 3.7]{cluster4}, it is enough to prove the result for principal coefficients. Let $\Gamma$ be any triangulation of $(S,M)$ and let $(Q_\Gamma,W_\Gamma)$ be the quiver with potential associated with this triangulation as in \cite{ABCP} (see also \cite{Labardini:potentialssurfaces}). Let $\mathcal J_\Gamma = \mathcal J_{(Q_\Gamma,W_\Gamma)}$ be the corresponding Jacobian algebra which is finite dimensional \cite{ABCP}. It follows from Corollary \ref{corol:glueing} that the principal extension $(Q_\Gamma^\pp, (Q_\Gamma)_0')$ can be embedded in a realisable quadruple $(Q_\Gamma^\pp, (Q_\Gamma)_0', \CC, T)$. Now, let $z$ be a cluster variable in $\mathcal A(S,M)$, it follows from Theorem \ref{theorem:FK} that either $z$ belongs to the initial cluster $\x_{Q_\Gamma}$ or $z = X^T_M$ for some indecomposable rigid $\mathcal J_\Gamma$-module $M$. Since the algebra $\mathcal J_\Gamma$ is a string algebra, any indecomposable rigid $\mathcal J_\Gamma$-module is a string module. Thus, if we set $y_i = x_{i'}$ for any $i \in (Q_T)_0$, it follows from Corollary \ref{corol:positivity} that
			$$z = X^T_M \in \Z_{\geq 0}[{\x_{\overline{\supp(M)}}}^{\pm 1}] \subset \Z_{\geq 0}[\y^{\pm 1}][\x_{Q_\Gamma}^{\pm 1}] = \Z_{\geq 0}\P[\x_{Q_\Gamma}^{\pm 1}].$$
		\end{proof}

	\subsection{$SL(2,\Z)$ and total Grassmannians of submodules}
		\begin{defi}
			Given a finite dimensional $\k$-algebra $B$ and a finitely generated $B$-module $M$, the \emph{total Grassmannian of submodules of $M$} is the set
			$$\Gr(M) = \bigsqcup_{\mathbf e \in K_0(\modd B)} \Gr_{\mathbf e}(M).$$
		\end{defi}
		Since $M$ is finite dimensional, all but finitely many $\Gr_{\mathbf e}(M)$ are empty so that $\Gr(M)$ is a finite disjoint union of projective varieties. It is thus endowed with a structure of projective variety and we can consider its Euler-Poincar\'e characteristic $\chi(\Gr(M)) = \sum_{\mathbf e} \chi(\Gr_{\mathbf e}(M)) \in \Z$.

		For modules over hereditary algebras, these numbers can be computed with friezes, as shown in \cite[\S 4]{AD:algorithm}. Here we prove that, for string modules over 2-Calabi-Yau tilted algebras, these numbers can be computed with products of matrices in $SL(2,\Z)$.

		If $Q$ is any quiver, for any $\beta \in Q_1$, we define two matrices in $SL(2,\Z)$~:
		$$a(\beta)=\left[\begin{array}{cc}
			1 & 0 \\
			1 & 1
		\end{array}\right]
		\textrm{ and }
		a(\beta^{-1})=\left[\begin{array}{cc}
			1 & 1 \\
			0 & 1
		\end{array}\right]$$
		and for any walk $c=c_1 \cdots c_n$ in $Q$, we set
		$$
		l_c = \left[\begin{array}{cc} 1,1 \end{array}\right] \left( \prod_{i=1}^n a(c_i) \right) \left[\begin{array}{c}1 \\ 1 \end{array}\right] \in \N.$$
		Finally, if $c$ is a walk of length 0, we set $l_c =2$.

		\begin{prop}\label{prop:chiGrM}
			For any string module $M$ over a finite dimensional algebra, we have
			$$\chi(\Gr(M)) = l_{\s(M)}.$$
		\end{prop}
		\begin{proof}
			Let $B$ be a finite-dimensional algebra with ordinary quiver $Q$ and let $M$ be a string $B$-module. We use the notations of Section \ref{ssection:coverings}. Let $\widetilde M$ be the representation of the blow-up $\widetilde {Q_M}$ of $Q$ along $M$. Then $\widetilde M$ is a string representation of $\widetilde {Q_M}$ and we have $l_{\s(M)} = l_{\s(\widetilde M)}$. Moreover, it follows from \cite[Theorem 1.2 (a)]{Haupt:string} that 
			$$\sum_{\mathbf e \in \N^{Q_0}} \chi(\Gr_{\mathbf e}(M)) = \sum_{\mathbf e \in \N^{Q_0}} \sum_{\mathbf f \in \phi^{-1}(\mathbf e)}\chi(\Gr_{\mathbf f}(\widetilde M)) = \sum_{\mathbf f \in \N^{(\widetilde{Q_M})_0}} \chi(\Gr_{\mathbf f}(\widetilde M))$$
			so that the characteristic of the total Grassmannian of $M$ equals the characteristic of the total Grassmannian of $\widetilde M$. Since Grassmannians of submodules of $\widetilde M$ only depends on the support of $\widetilde M$, it is enough to prove the proposition for indecomposable modules over path algebras of Dynkin quivers of type $\mathbb A$.

			Let thus $B$ be the path algebra of a Dynkin quiver $Q$ of type $\A$, let $\CC$ be the cluster category of $Q$ and let $T=\k Q$. Then $(Q,\emptyset,\CC,T)$ is a realisable quadruple such that $B=\End_{\CC}(T)$. According to Theorem \ref{theorem:main}, we have
			$$X^T_M = \frac{1}{\mathbf x_Q^{\n_M}}L_M.$$
			Consider the surjective morphism of $\Z$-algebras.
			$$p : \left\{\begin{array}{rcll}
				\mathcal L(\x_Q) & \fl & \Z \\
				x_i & \mapsto & 1 & \textrm{ for any }i \in Q_0.
			\end{array}\right.$$
			Then it follows from the definition of the cluster character that $p(X^T_M) = \chi(\Gr(M))$.

			We now prove that $p(L_M) = l_{\s(M)}$. If the string $\s(M)$ of $M$ is of length zero, $M$ is simple and thus
			$$\chi(\Gr(M)) = \chi(\Gr_{[0]}(M))+\chi(\Gr_{[M]}(M))=1+1=2=l_{\s(M)}.$$
			Otherwise, we write $\s(M)=c_1 \cdots c_n$ with $n \geq 1$. Then, for any $i \in \ens{1, \ldots, n}$, we have $p(A(c_i)) = a(c_i)$ and $p(V_c(s(c_i))) = p(V_c(t(c_i)))$ is the identity matrix. Since $p$ is a ring homomorphism, it induces a ring homomorphism at the level of matrices and thus $p(L_M) = p(L_c) = l_{c}$. Since $p(\frac{1}{\mathbf x_Q^{\n_M}})=1$, it follows that
			$$\chi(\Gr(M)) = p(X^T_M) = p(L_M) = l_{\s(M)}$$
			and the proposition is proved.
		\end{proof}

		\begin{exmp}
			Let $(S,M)$ be an unpunctured surface. Let $\Gamma$ be a triangulation of $(S,M)$, let $(Q_\Gamma,W_\Gamma)$ be the quiver with potential associated with this triangulation and let $\mathcal J_\Gamma = \mathcal J_{(Q_\Gamma,W_\Gamma)}$ be the corresponding Jacobian algebra. Then for any string module $M$ over $\mathcal J_\Gamma$, we have $\chi(\Gr(M)) = l_{\s(M)}$.
		\end{exmp}

% 		\begin{rmq}
% 			If a marked surface $(S,M)$ has punctures, Labardini also associated a quiver with potential $(Q_\Gamma,W_\Gamma)$ to any triangulation $\Gamma$ of $(S,M)$ \cite{Labardini:potentialssurfaces}. Nevertheless, it may be that the Jacobian algebra $\mathcal J_{(Q_\Gamma,W_\Gamma)}$ is infinite dimensional, in which case Theorem \ref{theorem:main} may not be used any longer. It was nevertheless observed by Lasnier that, in the case of the one-punctured torus (where the Jacobian algebra $\mathcal J_\Gamma$ is infinite dimensional), Euler-Poincar\'e characteristics of total Grassmannians of submodules of string $\mathcal J_\Gamma$-modules may still be computed with a formula similar to $M \mapsto l_{\s(M)}$ \cite{Lasnier:thesis}. 
% 		\end{rmq}

\section{More about the normalisation}\label{section:morenormalisation}
	We now describe some situations in which the normalisation can be omitted or computed combinatorially.

	\subsection{The hereditary case}
		We first observe that the normalisation can be omitted in several cases.

		\begin{lem}\label{lem:hereditary}
			Let $(\mathcal Q,F,\CC,T)$ be a realisable quadruple and let $M$ be an unfrozen string module. If the full subcategory of mod-$B_T$ formed by modules which are supported on $\overline{\supp(M)}$ is hereditary, then $\mathbf n_M = 0$.
		\end{lem}
		\begin{proof}
			Let $i \in \overline{\supp(M)}_0$, since $M$ and $S_i$ are supported on $\overline{\supp(M)}$, the truncated Euler form $\<S_i,M\>$ only depends on the dimension vectors. But since $\widetilde{Q_M}$ is acyclic, $\sum_{j \in \Phi_0^{-1}(i)}\<S_j,\widetilde M\>$ also only depends on the dimension vectors, it thus follows that $\<S_i,M\> = \sum_{j \in \Phi_0^{-1}(i)}\<S_j,\widetilde M\>$ and thus $\n_M=0$.
		\end{proof}

		\begin{corol}
			Let $(\mathcal Q,F,\CC,T)$ be a realisable quadruple such that $B_T$ is hereditary. Then $\mathbf n_M=0$ for every unfrozen string module $M$. \hfill \qed
		\end{corol}

	\subsection{The maximal factor method for modules over cluster-tilted algebras}
		For any walk $c$ in a locally finite quiver $Q$ we set
		$$N_c = \left[\begin{array}{cc} 1,1 \end{array}\right] \left( \prod_{i=0}^n A(c_i) V_{c}(i+1)\right) \left[\begin{array}{c}1 \\ 1 \end{array}\right] \in \mathcal L(\x_Q)$$
		so that \eqref{eq:Lc} becomes
		$$L_c = \frac{1}{\prod_{i=0}^n x_{t(c_i)}} N_c.$$
		Let $\eta_c \in \N^{Q_0}$ be such that
		$$N_c = \x_Q^{\eta_c} P_c(\x_Q)$$
		with $P_c(\x_{Q})$ not divisible by any $x_i$ with $i \in Q_0$. % The factor $\x_Q^{\eta_c}$ is called the \emph{maximal factor} of $c$.

		Let $\CC$ be a cluster category, $T=T_1 \oplus \cdots \oplus T_n$ be a tilting object in $\CC$, then $B_T = \End_{\CC}(T)$ is a cluster-tilted algebra and we denote by $Q$ its ordinary quiver. Then $(Q, \emptyset, \CC,T)$ is a realisable quadruple and every $B_T$-module is unfrozen. % We recall that a $B_T$-module is called \emph{rigid} if $\Ext^1_{B_T}(M,M)=0$.

		We now prove that for such modules, we can compute the normalisation combinatorially~:
		\begin{prop}\label{prop:combnorm}
			Consider the above notations and assume moreover that $\End_{\CC}(T_i) \simeq \k$ for any $i \in \ens{1, \ldots, n}$. Then for every rigid string $B_T$-module $M$ we have $\n_M=\eta_{\s(M)}$.
		\end{prop}
		\begin{proof}
			For any Laurent polynomial $L(x_1, \ldots, x_n) \in \Z[x_1^{\pm 1}, \ldots, x_n^{\pm 1}]$, we denote by $\delta(L) \in \Z^{n}$ its \emph{denominator vector}, that is, the unique vector $(d_1, \ldots, d_n)$ such that
			$L(x_1, \ldots, x_n)=f(x_1, \ldots, x_n)/\prod_{i=1}^n x_i^{d_i}$ where $f(x_1, \ldots, x_n)$ is a polynomial not divisible by any $x_i$.
 			
			Since $\End_{\CC}(T_i) \simeq \k$ for any $i \in \ens{1, \ldots, n}$, it follows from \cite{BMR3} that for any indecomposable rigid $B_T$-module
			$$\delta(X^T_M) = (\dim \Hom_{\CC}(T_i,\overline M))_{1 \leq i \leq n}$$
			where $\overline M$ is an indecomposable rigid object in $\CC$ such that $\Hom_{\CC}(T,\overline M) = M$.
		
			Since $\Hom_{\CC}(T,-)$ induces an equivalence of categories $\CC/(\add T[1]) \simeq \modd B_T$, it follows that
			$$\delta(X^T_M) = (\dim \Hom_{B_T}(P_i,M))_{1 \leq i \leq n}$$
			where $P_i$ is the indecomposable projective $B_T$-module at the point $i$.
			For any $i \in \ens{1, \ldots, n}$, the dimension of $\Hom_{B_T}(P_i,M)$ is the multiplicity of the simple $B_T$-module $S_i$ as a composition factor of $M$. If $M$ is a string module, this multiplicity is also equal to the number of occurrences of $i$ along the string $\s(M)$. Thus, if $\s(M)=c_1 \cdots c_n$, the denominator of $X^T_M$ in its irreducible form is equal to $\prod_{i=0}^n x_{t(c_i)}$.
	
			On the other hand the denominator of $L_M$ (in its irreducible form) is $\x^{-\eta_{\s(M)}} \prod_{i=0}^n x_{t(c_i)}$. Now it follows from Theorem \ref{theorem:main} that $X^T_M = \frac{1}{\x^{\n_M}} L_M$ and thus the denominators coincide, that is to say $\n_M=\eta_{\s(M)}$.
		\end{proof}

% 		\begin{rmq}
% 			If $\CC$ is the cluster category of a euclidean quiver, we can slightly modify the above proposition in order to get a similar result even if $\End_{\CC}(T_i) \not \simeq \k$. For this, we use \cite{BM:affine} instead of \cite{BMR3}. Details are left to the reader.
% 		\end{rmq}

\section{Examples}\label{section:examples}
	\subsection{A first example in type $\A_2$}
		We start with a very simple and detailed example which is nevertheless instructive in order to understand the behaviour of the normalisation.

		We consider the frozen quiver $(\mathcal Q,F)$ where
		$$
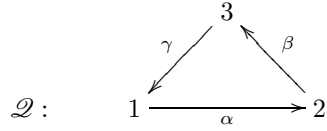
\xymatrix{
				&& 3 \ar[ld]_{\gamma}\\
			\mathcal Q : & 1 \ar[rr]_{\alpha} && 2 \ar[lu]_{\beta}
		}$$
		and $F=\ens{3}$.

		The corresponding cluster algebra is of geometric type with initial seed $(\x,\y,B)$ where $\x=(x_1,x_2)$, $\y = (x_3)$ and
		$$B= \left[\begin{array}{rr}
			0 & 1 \\
			-1 & 0 \\
			1 & -1
		\end{array}\right]$$

		Its unfrozen part $Q$ is of Dynkin type $\A_2$. Note that $\mathcal Q$ is the ordinary quiver of a cluster-tilted algebra of type $\A_3$ whose relations are $\alpha \beta = \beta\gamma = \gamma \alpha = 0$. Thus $(\mathcal Q,F)$ can be embedded in a realisable quadruple $(\mathcal Q,F,\CC,T)$ where $\CC$ is a cluster category of type $\A_3$. There are three indecomposable modules supported on $Q$ which are the simple $S_1$, the simple $S_2$ and the projective $\k Q$-module $P_1$ and all these modules are string modules.

		Since $(\mathcal Q,F)$ is not blown-up (there are two arrows entering or leaving the frozen point 3), we cannot simply apply Theorem \ref{theorem:typeA}. So we apply Theorem \ref{theorem:main}, which means that we need to consider normalising factors.

		The string of $S_1$ is $e_1$ and thus, applying formula \eqref{eq:Lc} to $S_1$, we get $L_{S_1} = \frac{x_2 + x_3}{x_1}$. The quiver $\widetilde{Q_{S_1}}$ is $3^{\gamma;1} \fl v_1 \fl 2^{\alpha;1}$. It is easily computed that $\n_{S_1}=0$ so that $X_{S_1} = L_{S_1}$ which is indeed the cluster variable in $\mathcal A(\mathcal Q,F)$ corresponding to the simple module $S_1$. We similarly prove that $X_{S_2} = L_{S_2} = \frac{x_1 + x_3}{x_2}.$

		The case of $P_1$ is slightly more instructive. We apply the matrix formula and we get $L_{P_1} = \frac{x_3(x_1+x_2+x_3)}{x_1x_2}$. We can now use Proposition \ref{prop:combnorm} in order to conclude that the normalising factor is $x_3$ but we do the computation in order to see the blow-up technique working. The quiver $\overline{\supp(P_1)}$ is $\mathcal Q$ itself. Its blow-up along $P_1$ is thus $\widetilde {Q_{P_1}} : 3^{\gamma;1} \fl v_1 \fl v_2 \fl 3^{\beta;2}$ and $\widetilde{M}$ is the representation $\widetilde P_1 : 0 \fl \k \xrightarrow{1_\k} \k \fl 0$. We compute that
		$$\widetilde{L}_{\widetilde P_1} = \frac{x_{v_1}x_{3^{\gamma;1}}+x_{v_2}x_{3^{\beta;2}}+x_{3^{\gamma;1}}x_{3^{\beta;2}}}{x_{v_1}x_{v_2}}$$
		so that the morphism $\pi$ sending $x_{v_1}$ to $x_1$, $x_{v_2}$ to $x_2$ and $x_{3^{\gamma;1}},x_{3^{\beta;2}}$ to $x_3$ satisfies $\pi(\widetilde{L}_{\widetilde M}) = L_M$. And computing directly, or else applying Theorem \ref{theorem:typeA}, we get $\widetilde{X}_{\widetilde M} = \widetilde{L}_{\widetilde P_1}$.

		The normalisation of $M$ is $\n_M = (0,0,1)$ (see Example \ref{exmp:normalisation}) so that we get $X_{P_1} = \frac{1}{x_3} L_{P_1} = \frac{x_1+x_2+x_3}{x_1x_2}$ which is indeed the cluster variable in $\mathcal A(\mathcal Q,F)$ corresponding to this module.

	\subsection{Regular cluster variables in type $\widetilde \A$}
		Let $Q$ be a euclidean quiver of type $\widetilde \A$ equipped with an acyclic orientation. In \cite[Theorem 4]{ARS:frises}, the authors provided a combinatorial formula for expressing all but finitely many cluster variables in the coefficient-free cluster algebra associated with $Q$. The cluster variables they computed correspond in fact to the cluster characters associated with indecomposable postprojective $\k Q$-modules. Using their methods, it is possible to compute the cluster characters associated with indecomposable postprojective and preinjective $\k Q$-modules but not the remaining cluster variables, namely those which correspond to indecomposable regular rigid $\k Q$-modules. We now use Theorem \ref{theorem:typeA} to complete the formula and compute cluster variables associated with regular modules.

		It is known that the Auslander-Reiten quiver of mod-$\k Q$ contains at most two exceptional tubes (that is, tubes of rank $\geq 2$) and all the rigid indecomposable regular modules belong to these tubes. Every indecomposable regular module (rigid or not) in such a tube is a string module and its string may be completely described. For simplicity, we only compute the cluster variables associated with quasi-simple regular modules in such tubes, this can easily be extended to any module in one of the exceptional tubes. It is well-known (see for instance \cite[\S 5.2]{AD:algorithm}) that any such module is of the form
		$$\xymatrix{
			\cdots 0 & \ar[l]_0 \k \ar[r]^{1_k} & \k \ar[r]^{1_k} & \cdots \ar[r]^{1_k} & \k \ar[r]^{1_k} & \k & \ar[l]_0 0 \cdots
		}$$
		so that, if we depict locally the quiver $Q$ as
		$$\xymatrix{
			\cdots 0 & \ar[l] 1 \ar[r]^{\alpha_1} & 2 \ar[r]^{\alpha_2} & \cdots \ar[r]^{\alpha_{n-2}} & n-1 \ar[r]^{\alpha_{n-1}} & n & \ar[l] n+1 \cdots ,
		}$$
		the string is $c = \alpha_1 \cdots \alpha_{n-1}$ and we compute
		\begin{align*}
		L_c
			& = \frac{1}{\prod_{i=1}^{n+1}x_i}
			\left[\begin{array}{cc} 1,1 \end{array}\right]
			\left[\begin{array}{cc}
				x_0 & 0 \\
				0 & 1
			\end{array}\right]
			\left[\begin{array}{cc}
			\displaystyle \prod_{i=1}^{n-1} x_{i+1} & 0 \\
				\displaystyle \sum_{j=1}^{n-1} \frac{\prod_{i=1}^n x_i}{x_jx_{j+1}} & \displaystyle \prod_{i=1}^{n-1} x_{i}
			\end{array}\right]
			\left[\begin{array}{cc}
				1 & 0 \\
				0 & x_{n+1}
			\end{array}\right]
			\left[\begin{array}{c}1 \\ 1 \end{array}\right].
		\end{align*}
		Now, if $M_c$ denotes the quasi-simple regular module associated with the string $c$, then it follows from Lemma \ref{lem:hereditary} that $L_c$ is the cluster character corresponding to the module $M_c$ and thus, is a cluster variable in the cluster algebra $\mathcal A(Q,\emptyset)$.

	\subsection{An example for the $n$-Kronecker quiver with principal coefficients}
		Let $n \geq 2$ and let $K_n^{\pp}$ be the quiver
		$$\xymatrix{
					& 1' \ar[d] & 2' \ar[d] \\
			K_n^{\pp} : 	& 1 \ar@<+3pt>[r]^{\alpha_1} \ar@{..>}[r] \ar@<-3pt>[r]_{\alpha_n} & 2
		}$$
		with $n$ arrows $\alpha_1, \ldots, \alpha_n$ from 1 to 2. We set $F=\ens{1',2'}$ so that $K_n^{\pp}$ is the principal extension of the $n$-Kronecker quiver $K_n$. We fix two distinct arrows among $\alpha_1, \ldots, \alpha_n$ which we denote by $\alpha$ and $\beta$.

		For any $p \geq 1$, we define a string representation $M^p$ of $K_n$ by setting $M^p(1) = \k^p$, $M^p(2) = \k^{p+1}$, $M^p(\alpha) = 1_{\k^p} \oplus 0$ and $M^p(\beta) = 0 \oplus 1_{\k^p}$ and we view $M^p$ as a representation of $K_n^{\pp}$ which is supported on the unfrozen part $K_n$. For any $p \geq 1$, the string $\s(M^p)$ is $(\alpha^{-1}\beta)^p$ so that it has length $2p$.

		We write $y_1=x_{1'}$ and $y_2=x_{2'}$. Then, a direct computation shows that
		\begin{equation}\label{eq:LMp}
		L_{M^p} = \frac{1}{x_1^px_2^{p+1}} \left[\begin{array}{cc} 1,1 \end{array}\right]
		\left[\begin{array}{cc}
			1 & 0 \\
			0 & y_2x_1^{n-1}
		\end{array}\right]
		\left[\begin{array}{cc}
			y_1 + x_2^n & y_1 y_2x_1^{n-1} \\
			y_1x_1 & y_1y_2x_1^n
		\end{array}\right] ^p	
		\left[\begin{array}{c}1 \\ x_1 \end{array}\right].
		\end{equation}

		Applying Corollary \ref{corol:glueing}, we see that there exists a realisable quadruple $(K_n^{\pp}, (K_n)_0',\CC,T)$. Since $K_n^{\pp}$ is acyclic, it follows from \cite{KR:acyclic} that $B_T$ is hereditary and thus, Lemma \ref{lem:hereditary} implies that the normalisation vanishes. Thus, $X^T_{M^p} = L_{M^p}$ is given by the formula \eqref{eq:LMp} for every $p \geq 1$.

	\subsection{Cyclic cluster-tilted algebras of Dynkin type $\mathbb D$}
		Let $B$ be the quotient of the path algebra of the quiver
		
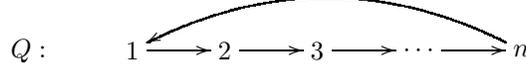
\vspace{5mm}
		\[ \xymatrix{Q: & 1\ar[r]&2\ar[r]&3\ar[r]&\cdots\ar[r]&n\ar@/_20pt/[llll]}\]
		by the ideal generated by all paths of length $n-1$. Here we suppose that $n\ge 3$. $B$ is a cluster-tilted algebra of type $\mathbb{D}_n$ if $n\ge 4$ and a cluster-tilted algebra of type $\mathbb{A}_3 $ if $n=3$.
		Let $M$ be the indecomposable module given by the Loewy series
		$\begin{array}{c} 2 \vspace{-3pt} \\3 \vspace{-3pt}
		\\ \vdots  \vspace{-3pt}\\ m \vspace{-3pt} \end{array}$,
		where $2\le m\le n$.

		\begin{prop}
			Let $Q$ denote the above quiver and let $(\mathcal Q,F,\CC,T)$ be a realisable quadruple with unfrozen part $Q$. For any $i \in Q_0$, let, as in Section \ref{ssection:ARS},  \[
			y_i=\displaystyle\prod_{\alpha \in \mathcal Q_1(F,i)}x_{s(\alpha)} \quad \text{and} \quad z_i=\displaystyle\prod_{\alpha \in \mathcal Q_1(i,F)}x_{t(\alpha)}.
			\]
Let $M$ be as above. Then
			\begin{equation}\label{exfo1}
				L_M = \sum_{\ell=1}^m \frac{x_1x_2\cdots x_{m+1}}{x_\ell x_{\ell+1}} \left( \prod_{i=2}^\ell y_i \right)
				\left(\prod_{i=\ell+1}^m z_i \right) \frac{1}{\prod_{j=2}^m x_j}
			\end{equation}
			if $m<n$; and if $n=m$ then $L_M$ is given by \eqref{exfo1} divided by $x_1$.
		\end{prop}

		\begin{proof}
		We proceed by induction on $m$. If $m=2$, then $M$ is simple and $L_M$ is given by the exchange relation
		\begin{equation}\label{exfo2}
			L_M=\left({x_1y_2 + x_3 z_2}\right)\frac{1}{x_2}.
		\end{equation}
		On the other hand, (\ref{exfo1}) gives
		\[\left(\frac{x_1x_2x_3}{x_1x_2} z_2
		+\frac{x_1x_2x_3}{x_2x_3} y_2 \right) \frac{1}{x_2}, \]
		which is equal to the right-hand side of equation (\ref{exfo2}).

		Suppose now that $m>2$. Let $N_M$ be the numerator of the Laurent polynomial $L_M$.
		Our formula gives $N_M$ as the product of $2m-1$ matrices :
		\begin{equation}\label{exfo3}
			\left[ \begin{array}{cc} 1&x_1  \end{array} \right]
			\left[ \begin{array}{cc} z_2&0\\ 0 & y_2  \end{array} \right]
			\cdots
			\left[ \begin{array}{cc} x_m&0\\1 &x_{m-1}  \end{array} \right]
			\left[ \begin{array}{cc} z_m&0\\0 & y_m  \end{array} \right]
			\left[ \begin{array}{cc} x_{m+1}\\1 \end{array} \right].
		\end{equation}
		The product of the last three matrices in \eqref{exfo3} is equal to
		\[
		\left[ \begin{array}{cc}  x_mx_{m+1} z_m\\
		x_{m+1} z_m \ +\ x_{m-1} y_m  \end{array}  \right]
		=
		\left[ \begin{array}{cc}  x_m\\
		1\end{array}\right]x_{m+1} z_m  \ + \
		\left[\begin{array}{cc} 0\\1 \end{array}  \right]  x_{m-1} y_m  .
		\]
		Let us denote the product of the first $2m-4$ matrices by $
		\left[ \begin{array}{cc} a&b  \end{array} \right]$, so
		\[N_M= \left[ \begin{array}{cc} a&b  \end{array} \right] \left(
		\left[ \begin{array}{cc}  x_m\\
		1\end{array}\right]x_{m+1} z_m  \ + \
		\left[\begin{array}{cc} 0\\1 \end{array}  \right]  x_{m-1} y_m  \right).
		\]
		Now let $M'$ be the indecomposable module given by the Loewy series
		$\begin{array}{c} 2 \vspace{-3pt} \\3 \vspace{-3pt}
		\\ \vdots  \vspace{-3pt}\\ m-1 \vspace{-3pt} \end{array}$. Then, our formula gives $N_{M'}$ as a product of $2m-3$ matrices, and the first $2m-4$ matrices are just the same as the first $2m-4$ matrices in (\ref{exfo3}), and therefore
		\[ N_{M'} = \left[ \begin{array}{cc} a&b  \end{array} \right]  \left[ \begin{array}{cc} x_m\\1  \end{array} \right]. \]
		Thus
		\[ N_M =N_{M'}\,  x_{m+1} z_m  \ + \ \left[ \begin{array}{cc} a&b  \end{array} \right]
		\left[\begin{array}{cc} 0\\1 \end{array}  \right]  x_{m-1} y_m ,
		\]
		which, by induction, is equal to
		\[ 
%\begin{array}{rl}& 
\displaystyle\sum_{\ell =1}^{m-1} \frac{x_1x_2\cdots x_m}{x_\ell x_{\ell+1} }
		\left( \prod_{i=2}^\ell y_i \right)
		\left(\prod_{i=\ell+1}^{m-1}z_i \right) x_{m+1} z_m
%		\\
		+
%&  
\displaystyle\frac{x_1x_2\cdots x_m}{x_{m-1} x_{m} }
		\left( \prod_{i=2}^{m-1} y_i \right)
		x_{m-1} y_m  ,
%		\end{array}
		\]
		which in turn is equal to
		\[\sum_{\ell =1}^{m-1} \frac{x_1x_2\cdots x_{m+1}}{x_\ell x_{\ell+1} }
		\left( \prod_{i=2}^\ell y_i \right)
		\left(\prod_{i=\ell+1}^{m}z_i \right)
		+ \frac{x_1x_2\cdots x_{m+1}}{x_{m} x_{m+1} }
		\left( \prod_{i=2}^{m} y_i \right) ,
		\]
		and this shows the formula (\ref{exfo1}).
		The statement for $m=n$ follows from the normalising factor.
		\end{proof}

\section*{Acknowledgments}
	Work on this problem was started during the \emph{2010 South American Meeting on Representations of Algebras and Related Topics} in Mar del Plata (Argentina); the authors wish to thank the organisers for their kind invitation.

	The first author gratefully acknowledges partial support from the NSERC of Canada, the FQRNT of Qu\'ebec and the Universit\'e de Sherbrooke.

	This paper was written while the second author was at the Universit\'e de Sherbrooke as a CRM-ISM postdoctoral fellow under the supervision of the first author, Thomas Br\"ustle and Virginie Charette.

	The third author is supported by the NSF grants DMS-0908765 and DMS-1001637 and by the University of Connecticut.

	The fourth author gratefully acknowledges support from the NSERC of Canada and Bishop's University.

% \bibliographystyle{abbrv}
% \bibliography{../../biblio}

\end{document}